\documentclass[preprint,12pt]{elsarticle}

\usepackage{amsmath, amsthm, euscript,  enumerate}
\usepackage{epsfig}

\usepackage{amssymb}

\usepackage{epsfig}
\usepackage{color}

\newcommand{\R}{{\mathbb R}}
\newcommand{\Z}{{\mathbb Z}}

\newcommand{\e}{\epsilon}
\newcommand{\vp}{\varphi}
\newcommand{\osc}{\operatornamewithlimits{osc}}

\newcommand{\supp}{\operatorname{supp}}
\newcommand{\D}{\nabla}
\newcommand{\ra}{\rightarrow}
\newcommand{\La}{\triangle}
\newcommand{\bs}{\backslash}
\newcommand{\wl}{\overset{w}{\longrightarrow}}

\newtheorem{theo}{Theorem}[section]
\newtheorem{lemm}[theo]{Lemma}

\newtheorem{remark}[theo]{Remark}
\newtheorem{corollary}[theo]{\textbf{Corollary}}
\numberwithin{equation}{section}

\begin{document}

\begin{frontmatter}



\title{Viscosity method for Homogenization of Parabolic Nonlinear Equations in Perforated
Domains}


\author{Sunghoon Kim \corref{cor}}

\address{School of Mathematical Sciences,
Seoul National University, San56-1 Shinrim-dong Kwanak-gu Seoul
151-747, South Korea} \ead{ gauss79@snu.ac.kr }
\cortext[cor]{Corresponding author.}

\author{Ki-Ahm Lee}
\address{School of Mathematical Sciences,
Seoul National University, San56-1 Shinrim-dong Kwanak-gu Seoul
151-747, South Korea } \ead{ kiahm@math.snu.ac.kr }


\begin{abstract}
In this paper, we develop a viscosity method for  Homogenization of
Nonlinear Parabolic Equations constrained by highly oscillating
obstacles or Dirichlet data in perforated domains. The Dirichlet
data on the perforated domain can be considered as a constraint or
an obstacle. Homogenization of nonlinear eigen value problems has
been also considered to control the degeneracy of the Porous medium
equation in perforated domains. For the simplicity, we consider
obstacles that consist of cylindrical columns distributed
periodically and perforated domains with punctured balls. If the
decay rate of the capacity of columns or the capacity of punctured
ball is too high or too small, the limit of $u_{\e}$ will converge
to  trivial solutions. The critical decay rates of having nontrivial
solution are obtained with the construction of barriers. We also
show the limit of $u_{\e}$ satisfies a homogenized equation with a
term showing the effect of the highly oscillating obstacles or
perforated domain in viscosity sense.
\end{abstract}


\end{frontmatter}

\section{introduction}\label{sec-intro}
This paper  concerns on  the homogenization of nonlinear parabolic
equations in perforated domains. Many physical models arising in the
media with a periodic structure will have solutions with
oscillations in the {\it micro scale}. The periodicity of the
oscillation denoted by $\e$ is much smaller compared to the size of
the sample in the media having {\it macro scale}. The presence of
slow and fast varying variables in the solution is the main obstacle
on the way of numerical investigation in periodic media. It is
reasonable to find asymptotic analysis of solutions as $\e$ goes to
zero and to study the macroscopic or averaged description. In the
mathematical point of view, the partial differential equation
denoted by $L_{\e}$ may have oscillation coefficients and even the
domains, $\Omega_{\e}$, will have periodic structure like a
perforated domain. So for each of $\e>0$,we have solutions $u_{\e}$
satisfying
  $$L_{\e}u_{\e}=0 \quad \textrm{in $\Omega_{\e}$},$$
with an appropriate boundary condition. It is an important step to
find the sense of convergence of $u_{\e}$ to a limit $u$ and the
equation called {\it Homogenized Equation}
$$\overline{L} u=0\quad\textrm{in $\Omega$}$$
satisfied by $u$. Such process is called {\it Homogenization}.

Large number of literatures on this topic can be found in
\cite{BLP},\cite{JKO}. And various notion of convergences have been
introduced, for example $\Gamma$-convergence of DeGiorgi, \cite{DG},
$G$-convergence of Spagnolo ,\cite{S}, and  $H$-convergence of
Tartar, \cite{T1}.  Two-scale asymptotic expansion method has been
used to find $\overline{L}$ formally and justified by the energy
method of Tartar.  He was able to pass the limit through {\it
compensated compactness} due to a particular choice of oscillating
test function \cite{T}. For the periodic structure, two-scale
convergence was introduced by Nguetseng, \cite{N} and Allaire,
\cite{A}, which provides the convergence of $u_{\e}(x)$ to a
two-scale limit $u_0(x,y)$ in self-contained fashion. And recently
viscosity method for homogenization has been developed by Evans,
\cite{Ev1} and Caffarelli, \cite{Ca}. Nonvariational problems in
Homogenization has been considered in \cite{CL1},\cite{CL2}. They
observe that the  homogenization of some parabolic flows could be
very different from the homogenization process by Energy method. For
example, there could be multiple solutions in reaction diffusion
equations. It is noticeable that the parabolic flows with initial
data  larger than largest viscosity elliptic solution will never
cross the stationary viscosity solution and that the homogenization
will happen away from a stationary solution achieved by minimizing
the corresponding  energy, \cite{CLM}\cite{CLM1}.  And the viscosity
method has been applied to the homogenization of nonlinear partial
differential equations with random data \cite{CSW},\cite{CLM2}.

Now let us introduce an example of parabolic equations in perforated
domains. Set $\Omega$ be a bounded connected subset of
$\mathbb{R}^n$ with smooth boundary. We are going to obtain a
perforated domain. For each $\e>0$, we cover $\mathbb{R}^n$ by cubes
$\cup_{m\in \epsilon \mathbb{Z}^n}Q^{\epsilon}_m$ where a cube
$Q^{\epsilon}_m$ is centered at $m$ and is of the size $\epsilon$.
Then from each cube, $Q^{\e}_m$, we remove a ball, $T_{a_{\e}}(m)$,
of radius $a_{\e}$ having the same center of the cube $Q^{\e}_m$.
Then we can produce a domain that is perforated by spherical
identical holes. Let
\begin{equation*}
\begin{aligned}
T_{a_{\epsilon}}&:=\cup_{m\in\epsilon\Z^{n}}T_{a_{\epsilon}}(m), \\
\mathbb{R}^n_{a_{\e}}&:=\mathbb{R}^n \bs T_{a_{\epsilon}}
\end{aligned}
\end{equation*}
and
\begin{equation*}
\begin{aligned}
\Omega_{a_{\e}}&:=\Omega\cap \mathbb{R}^n_{a_{\e}}=\Omega \bs T_{a_{\e}},\\
Q_{T,a_{\e}}&:=\Omega_{a_{\e}}\times (0,T].
\end{aligned}
\end{equation*}
Now  we are going to construct the highly oscillating obstacles. Let
us consider a  smooth function $\varphi(x,t)$ in
$Q_T=\Omega\times(0,T]$ which is negative on the lateral boundary
$\partial_lQ_T$, i.e. $\varphi \leq 0$ on $\partial_lQ_T$ and
positive in some region of $Q_T$. Highly oscillating obstacle
$\vp_{\e}(x,t)$ is zero in $\Omega_{a_{\e}}$ and $\vp(x,t)$ on each
hole $T^{a_{\e}}_m$:
\begin{equation*}
\begin{aligned}
\varphi_{\e}&:=\varphi \chi_{T_{a_{\epsilon}}}\\
&=\begin{cases} \varphi(x,t) \qquad \qquad \textrm{if} \quad
(x,t) \in T_{a_{\epsilon}}\times(0,T]\\
\quad 0 \qquad \qquad \qquad \qquad \textrm{Otherwise.}
\end{cases}
\end{aligned}
\end{equation*}
Then $\vp_{\e}(x,t)$ will oscillate more rapidly  between $0$ and
$\vp(x,t)$ as $\e$ goes to zero.

We can consider the standard obstacle problem asking the least
viscosity super-solution of Heat operator above the given
oscillating obstacle: find the smallest viscosity super-solution
$u_{\epsilon}(x,t)$ such that
\begin{equation}
\begin{cases}
\begin{aligned}
H[u]=\La u_{\epsilon}&-u_t \leq 0 \qquad \qquad \mbox{in $Q_T$} \quad (=\Omega \times (0,T])\\
u_{\epsilon}(x,t) &=0\qquad \qquad \qquad  \mbox{on $\partial_l Q_T$} \quad (=\partial \Omega \times (0,T])\\
u_{\epsilon}(x,t)&\geq \vp_{\epsilon}(x,t)\qquad \qquad \, \mbox{in $Q_T$}\\
u_{\epsilon}(x,0)&=g(x) \qquad \qquad \quad \mbox{on $\Omega \times
\{0\}$}
\end{aligned}
\end{cases}
\tag{{\bf $H_{\e}$}} \label{eq-main}
\end{equation}
where $g(x)\geq \vp(x,0)$, $\vp_{\e}(x,t) \leq 0$ on $\partial_lQ_T$
and $\vp_{\e}$ is positive in some region of  $Q_T$. The concept of
viscosity solution and its regularity can be found at \cite{CC}.\\
\indent We are interested in the limit of the $u_{\e}$ as $\e$ goes
to zero. Then there are three possible cases. First, if the decay
rate $a_{\e}$ of the radius of column is too high  w.r.t. $\e$, the
limit solution will not notice  the existence of the obstacle. Hence
it will satisfy the Heat equation without any obstacle. Second, on
the contrary, if the decay rate $a_{\e}$ is too slow, the limit
solution will be influenced fully by  the existence of the obstacle
and then become a solution of the obstacle problem with the obstacle
$\vp(x)$. We are interested in the third case when the decay rate
$a_{\e}$ is critical  so that the limit solution will have partial
influence from  the obstacle. Then we are able to show that there is
a limiting configuration that becomes a solution for an operator
which has the original operator, i.e. Heat operator, and an
additional term that comes from the influence of the oscillating
obstacles. Naturally we ask what is the critical rate  $ a_{\e}^*$
of the size of the obstacle so that there is non-trivial limit
$u(x)$ of $u_{\e}(x)$ in the last case and what is the homogenized
equation satisfied by the limit function $u$.

 The elliptic variational inequalities with highly
oscillating obstacles were first studied by Carbone and Colombini,
\cite{CC}, and developed by De Giorgi, Dal Maso and Longo,
\cite{DDL}, Dal Maso and Longo, \cite{DL}, Dal Maso,
\cite{D1}\cite{D2}, H. Attouch and C. Picard, \cite{AP}, in more
general context. The energy method was considered by Cioranescu and
Murat, \cite{CM}.  The other useful references can be found in
\cite{CM}. The method of scale-convergence was adopted by J.
Casado-D\'{i}az  for nonlinear equation of $p$-Laplacian type in
perforated domain. and  the  parabolic version was studied by A.K.
Nandakumaran and M. Rajesh \cite{NR}. They considered the degeneracy
that is closed to parabolic $p$-Laplacian type and that doesn't
include the Porous Medium Equation type.  L. Baffico, C. Conca, and
M. Rajesh considered homogenization of eigen value problems in
perforated domain  for the nonlinear equation of $p$-Laplacian type,
\cite{BCR}.

  The obstacle problems for linear or nonlinear equation of the divergence type has been studied by many authors and the reference can be founded in \cite{Fr}.
The viscosity method for the obstacle problem of nonlinear  equation
of non-divergence type was studied by the author
\cite{L1},\cite{L2}.

 Caffarelli and Lee \cite{CL} develop a viscosity method for
the obstacle problem for Harmonic operator with highly oscillating
obstacles. This viscosity method is also improved into a fully
nonlinear uniformly elliptic operator homogeneous of degree one.

 The homogenization of highly oscillating obstacles for the Heat equation has been extended to the fully nonlinear equations of non-divergence type. This part is a parabolic version of the results in \cite{CL}. The same correctors constructed in \cite{CL} play an important role in the parabolic equation.
On the other hand, when we consider the following Porous Medium
Equation in perforated domain,  the viscosity method considered in
\cite{CL} can not be applied directly. The equation will be
formulated in the following form: find the viscosity solution
$u_{\e}(x,t)$ s.t.
\begin{equation}
\begin{cases}
\La u_{\epsilon}^m -\partial_tu_{\epsilon}=0 \qquad \qquad \mbox{in $Q_{T,a^{\ast}_{\epsilon}}\,\,(=\Omega_{a^{\ast}_{\epsilon}}\times(0,T])$}\\
u_{\epsilon} =0 \qquad \qquad \qquad \qquad \mbox{on $\partial_lQ_{T,a^{\ast}_{\epsilon}}\,\,(=\partial \Omega_{a^{\ast}_{\epsilon}}\times(0,T])$}\\
u_{\e}=g_{\epsilon} \qquad \qquad \qquad \qquad \mbox{on
$\Omega_{a^{\ast}_{\epsilon}}\times\{0\}$}
\end{cases}
\tag{{\bf $PME_{\e}^1$}} \label{eq-main1-1}
\end{equation}
with $1<m<\infty$ and a compatible $g_{\e}(x)$ which will be defined
at Section \ref{sec-por}. The Dirichlet boundary condition can be
considered an obstacle problem where the obstacle imposes the value
of the solution is zero in the periodic holes. And the diffusion
coefficient of ({\bf $PME_{\e}^1$}) is $mu^{m-1}$  and  will be zero
on $\partial \Omega_{a^{\ast}_{\epsilon}}$, which makes important
ingredients of the viscosity method for uniformly elliptic and
parabolic equations inapplicable without serious modification. Such
ingredients will be correctors, Harnack inequality, discrete
gradient estimate, and the concept of convergence. Therefore the
control of  the degeneracy of ({\bf $PME_{\e}^1$}) is a crucial part
of this paper.

One of the important observation is that
$U_{\epsilon}(x,t)=\frac{\varphi^{\frac{1}{m}}_{\epsilon}(x)}{(1+t)^{\frac{1}{m-1}}}$
will be a self-similar solution of ({\bf $PME_{\e}^1$}) if
$\varphi_{\epsilon}(x)$ satisfies the nonlinear eigen value problem:
\begin{equation*}
\begin{cases}
\begin{aligned}
\La \varphi_{\epsilon}&+\varphi_{\epsilon}^{\frac{1}{m}}=0 \qquad \qquad \qquad \mbox{in $\Omega_{a^{\ast}_{\epsilon}}$}\\
\varphi_{\epsilon}&=0 \qquad \qquad \qquad \qquad \mbox{on
$\partial\Omega_{a^{\ast}_{\epsilon}}$}.
\end{aligned}
\end{cases}
\end{equation*}
The equation  for $\varphi_{\epsilon}$ is uniformly elliptic with
nonlinear reaction term. The viscosity method in \cite{CL}, can be
applied to the homogenization of $\varphi_{\epsilon}$ with some
modification because of the nonlinearity of the reaction term. It is
crucial to capture the geometric shape of $\varphi_{\epsilon}$
saying that $\varphi_{\epsilon}$ is almost Lipschitz function with
spikes similar to the fundamental solution of the Laplace equation
in a very small neighborhood of the holes. It is not clear whether
we can find the geometric shape of $\varphi_{\epsilon}$ if we
construct $\varphi_{\epsilon}$ by the energy method since $H^1$-weak
solutions may have poor shapes. And then such  self-similar
solution, $U_{\epsilon}(x,t)$ will be used to construct super- and
sub-solution of ({\bf $PME_{\e}^1$})  in order to control the
solution, $u_{\epsilon}(x,t)$, especially the decay rate of
$u_{\epsilon}(x,t)$ as $x$ approaches to the holes, $T_{a_{\e}}$ at
Section \ref{sec-por}. With the help of such control, we are able to
prove the discrete gradient estimate of the $u_{\e}$ in order to
compare the values of $u_{\e}$ on a discrete lattice created
periodically by a point in a cell. And we also able to show the
almost flatness saying that the values of $u_{\e}$ at any two points
in each small cell are close to each other with an $\e$-error if
those points are away from the very small neighborhood of the hole
in the cell.

It is noticeable that the homogenized equation is expressed as a sum
between the original equation and a term depending on the capacity
and $(\varphi-u)_+$ as the case in  the heat equation, Theorem
\ref{thm-i}. We also prove that such decoupling of terms will happen
in the homogenization of porous medium equations in perforated
domain, Theorem \ref{thm-eigen}. But it is not clear whether such
decoupling property holds in the general fully nonlinear equations
of non-divergence type, Theorem \ref{thm-ii}.

\indent This paper is divided into three part: In Section
\ref{sec-obstacle}, we review some fact studied in \cite{CL}(highly
oscillating obstacle problem for Harmonic operator) and extend the
results of \cite{CL} to the Heat operator and fully nonlinear
parabolic operator.
 In Section \ref{sec-eigen}, we study the elliptic
eigenvalue problem in perforated domains, which describe the
behavior of solution of porous medium equations at a point close to
the boundary. And, in Section \ref{sec-por}, we deal with the estimates for the
porous medium equation in fixed perforated domain.
\\
\textbf{Notations:} Before we explain the main ideas of the paper,
let us summarize the notations and definitions that we will be used.
\begin{itemize}
\item $Q_T=\Omega\times(0,T]$, \quad
$\partial_lQ_T=\partial\Omega\times(0,T]$\\
\item $T_{a_{\epsilon}}$, $\mathbb{R}^n_{a_{\epsilon}}$, $\Omega_{a_{\epsilon}}$ and
$Q_{T,a_{\epsilon}}$ are described in Section \ref{sec-intro}.\\
\item We denote by $Q^{\e}_m$ the cube $\{x=(x_1,\cdots,x_n)\in \mathbb{R}^n:|x_i-m_i|\leq
\frac{\epsilon}{2}\,\, (i=1,\cdots,n)\}$ where
$m=(m_1,\cdots,m_n)\in \mathbb{R}^n$.\\
\item Denoting by $w_{\epsilon}$ the corrector described in
Section 2 in \cite{CL}.\\
\item Numbers: $a^{\ast}_{\epsilon}=\epsilon^{\alpha_{\ast}}$, $\alpha_{\ast}=\frac{n}{n-2}$ for
$n \geq 3$ and $a^{\ast}_{\epsilon}=e^{-\frac{1}{\epsilon^2}}$ for
$n=2$.
\end{itemize}
\section{Highly Oscillating Obstacle Problems}\label{sec-obstacle}
\indent First, we review results on the correctors in \cite{CL}.
Likewise Laplace operator in \cite{CL}, the correctors will be used
to correct a limit $u(x,t)$ of a solution $u_{\epsilon}(x,t)$ for the obstacle problems (\ref{eq-main}) in this section.\\
\indent Any possible limit, $u(x,t)$, can be corrected to be a solution of each $\epsilon$-problem,
(\ref{eq-main}), and it is also expected to satisfy a homogenized equation. The homogenized equation comes from a condition under which $u$ can be corrected to $u_{\epsilon}$. To have a oscillating corrector, let us consider a family
of functions, $w_{\epsilon}(x)$, which
satisfy
\begin{equation}\label{eq-cases-definition-of-corrector-cite-by-CL-5}
\begin{cases}
\La w_{\epsilon}=k \qquad \qquad \mbox{in $\mathbb{R}_{a_{\e}}^n=\R^n\bs T_{a_{\e}}$}\\
w_{\epsilon}(x)=1 \qquad \qquad  \mbox{in $T_{a_{\epsilon}}$}
\end{cases}
\end{equation}
for some $k>0$. In \cite{CL}, Caffarelli and one of authors
construct the super- or sub-solutions through which they  find the
limit of $w_{\epsilon}$
depending on the decay rate of the size of the oscillating obstacles, $a_{\epsilon}$.\\
\indent The next Lemma tells us that there is a critical rate
$a^{\ast}_{\epsilon}$ so that we get a nontrivial limit of
correctors, $w_{\epsilon}$. The reader can easily check following
the details in the proof of the Lemma 2.1 in \cite{CL}.
\begin{lemm}\label{lem-5}
Let $a_{\epsilon}=c_0\epsilon^{\alpha}$. There is a unique number
$\alpha_{\ast}=\frac{n}{n-2}$ s.t.
\begin{equation*}
\begin{cases}
\lim\inf w_{\epsilon}=-\infty, \qquad \qquad \mbox{for any $k>0$ if
$\alpha>\alpha_{\ast}$}\\
\lim\inf w_{\epsilon}=0 \quad \qquad \qquad \mbox{for $\alpha=\alpha_{\ast}$ and $k=cap(B_1)$}\\
\lim\inf w_{\epsilon}=1 \quad \qquad \qquad \mbox{for any $k>0$ if
$\alpha<\alpha_{\ast}$}.
\end{cases}
\end{equation*}
\end{lemm}
In addition, we can also obtain the interesting property from
\cite{JKO}.
\begin{lemm}
Set $\alpha=\alpha_{\ast}$, then the function
$\widehat{w}_{\epsilon}$ satisfying
\begin{equation*}
\widehat{w}_{\epsilon}=1-w_{\epsilon}
\end{equation*}
converges weakly to 1 in $H^1_{loc}(\mathbb{R}^n)$.
\end{lemm}

\subsection{Heat Operator}\label{sec-2-1}

\indent We are interested in the limit $u$ of the viscosity solution
$u_{\e}$ of ({\bf $H_{\e}$}) as $\e$ goes to zero and the
homogenized equation satisfied by the limit $u$. As we discussed in
the introduction, there will three possible cases depending on the
decay rate of $a_{\e}$.

\begin{theo}\label{theo-main-heat-operator-71930}
\label{thm-i}
\item Let $u_{\e}(x,t)$ be  the least viscosity super solution of \ref{eq-main}.
  \begin{enumerate}
   \item There is a continuous function $u$ such that $u_{\e}\wl u$ in $Q_T$ with respect to $L^{p}$-norm, for $p>0$.
       And for any $\delta>0$, there is a subset $D_{\delta}\subset Q_T$ and $\e_o$ such that , for $0<\e<\e_o$, $u_{\e}\ra u$ uniformly in $D_{\delta}$ as $\e\ra 0$ and $|Q_T\bs D_{\delta}|<\delta$.
   \item Let $a^*_{\e}=\e^{\alpha_*}$ for $\alpha_*=\frac{n}{n-2}$ for $n\geq 3$ and $a^*_{\e}=e^{-\frac{1}{\e^2}}$ for $n=2$.
     \begin{enumerate}
     \item  For $c_oa^*_{\e}\leq a_{\e}\leq C_oa^*_{\e}$, $u$ is             a viscosity solution of
            \begin{equation*}
            \begin{split}
            H[u]+&k_{B_{r_o}}(\vp-u)_+=0\quad\textrm{in $Q_T$}\\
                             &u=0\qquad \qquad \qquad \textrm{on $\partial_l Q_T$.}\\
                             u(x,&0)=g(x) \qquad \qquad \mbox{on $\Omega \times \{0\}$}
            \end{split}
            \label{eq-}
            \end{equation*}
            where $k_{B_{r_o}}$ is the harmonic capacity of $B_{r_o}$ if $r_o=\lim_{\e\ra 0} \frac{a_{\e}}{a^*_{\e}}$exists.
     \item If $a_{\e}=o(a^*_{\e})$ then $u$ is a viscosity solution of
           \begin{equation*}
            \begin{split}
                            H[u]&=0\qquad \qquad \quad \textrm{in $Q_T$}\\
                             u&=0\qquad \qquad \quad \textrm{on $\partial_lQ_T$.}\\
                             u(x,0)&=g(x) \qquad \qquad \mbox{on $\Omega \times \{0\}$}
            \end{split}
            \label{eq-la-thin}
            \end{equation*}
    \item If $a^*_{\e}=o(a_{\e})$ then $u$ is a least viscosity
          super solution of
          \begin{equation*}
            \begin{split}
                            H[u]&\leq 0\qquad \qquad \quad\textrm{in $Q_T$}\\
                            u &\geq \vp\qquad \qquad \quad\textrm{in $Q_T$}\\
                             u&=0\qquad \qquad \quad\textrm{on $\partial_lQ_T$.}\\
                             u(x,0)&=g(x) \qquad \qquad \mbox{on $\Omega \times \{0\}$}
            \end{split}
            \label{eq-la-thick}
            \end{equation*}
     \end{enumerate}
\end{enumerate}
\end{theo}
\begin{remark}
\begin{enumerate}
\item The boundary data above can be replaced by any smooth function. And $H[ u]=f(x,t)$ can replaced by the heat equation.
\item $T_{a_{\e}}=\{a_{\e}x: x\in D\}$ can be any domain with continuous boundary as long as there is two balls $B_{r_1}\subset D\subset B_{r_2}$
for $0<r_1\leq r_1<\infty$.   $B_{r_1}\subset D\subset B_{r_2}$ is
enough to construct super- and sub-solutions and then to find the
behavior of correctors, Lemma \ref{lem-5}. Then $k=cap(D)$ and $
k_{B_{r_o}}=k_D$.
\end{enumerate}
\end{remark}


\subsection{Estimates and Convergence} Every $\e$-periodic function
is constant on $\e$-periodic lattice $\e\Z^n$. The first observation
is that the  difference quotient of $u_{\e}$, instead of the first
derivative of $u_{\e}$, is uniformly bounded. The next important
observation is that a suitable scaled $u_{\e}$ is very close to a
constant multiple of a fundamental solution in a neighborhood of the
support of the oscillating obstacle, $T_{a_{\e}}$ and that $u_{\e}$
will be almost constant outside of it. These observations will be
proved in the following lemmas.

\subsubsection{ Estimates of $ u_{\e}$}\label{subsec-u-epsilon}

\begin{lemm}\label{lem-series-for-corollary-2-10----1}
For each unit direction $e\in \Z^n$, set
$$\Delta^{\epsilon}_{e}u_{\e}(x,t)=\frac{u_{\e}(x+\e
e,t)-u_{\e}(x,t)}{\e}.$$ Then
$$|\Delta^{\epsilon}_{e}u_{\e}(x,t)|<C$$ uniformly.
\label{lem-d-c11}
\end{lemm}

\begin{proof}
$u_{\e}$ can be approximated by the solutions, $u_{\e,\delta}$, of
the following penalized equations, \cite{Fr},
\begin{equation}
\begin{cases}
\begin{aligned}
-H[ u_{\e,\delta}](x,t)&+\beta_{\delta}(u_{\e,\delta}(x,t)-\vp_{\e}(x,t))=0\qquad\textrm{in $Q_T$}\\
 u_{\e,\delta}(x,t)&=0\qquad \qquad \qquad \qquad \quad \qquad
\qquad \textrm{on $\partial_l Q_T$}\\
u_{\e,\delta}(x,0)&=g(x) \qquad  \qquad \qquad \qquad \qquad
\mbox{on $\Omega\times\{0\}$}
\end{aligned}
\end{cases}
\label{eq-li-penalty}
\end{equation}
where the penalty term $\beta_{\delta}(s)$ satisfies
\begin{equation*}
\begin{array}{c}
\beta_{\delta}'(s)\geq 0,\quad \beta_{\delta}''(s)\leq 0,\quad \beta_{\delta}(0)=-1,\\
\beta_{\delta}(s)=0\quad \textrm{for $s>\delta$},\qquad
\beta_{\delta}(s)\ra -\infty\quad\textrm{for $s<0$}.
\end{array}
\label{eq-heat-s}
\end{equation*}
Let $Z=\sup_{(x,t)\in Q_T}|\Delta^{\epsilon}_{e}u_{\e ,\delta}|^2$
and assume that the maximum $Z$ is achieved at $(x_0,t_0)$. Then we
have , at $(x_0,t_0)$,
$$H[|\Delta^{\epsilon}_{e}u_{\e ,\delta}|^2]\leq 0,\quad\textrm{and}\quad \D |\Delta^{\epsilon}_{e}u_{\e ,\delta}|^2=0.$$
 By taking a difference quotient, we have
$$-H[\Delta^{\epsilon}_{e}u_{\e , \delta}]+\beta_{\delta}'(\cdot)(\Delta^{\epsilon}_{e}u_{\e,\delta}(x,t)-\Delta^{\epsilon}_{e}\vp_{\e}(x,t))=0.$$
Hence
\begin{equation*}
\begin{aligned}
-H[|\Delta^{\epsilon}_{e}u_{\e ,\delta}|^2]&+2|\nabla
(\Delta^{\epsilon}_{e}u_{\e
,\delta})|^2\\
&+2\beta_{\delta}'(\cdot)(|\Delta^{\epsilon}_{e}u_{\e,\delta}(x,t)|^2-\Delta^{\epsilon}_{e}u_{\e}(x,t)\Delta^{\epsilon}_{e}\vp_{\e}(x,t))=0.
\end{aligned}
\end{equation*}
Since the set $T_{a_{\e}}$ is $\e-$periodic and $\vp$ is $C^1$, we
know $|\Delta^{\epsilon}_{e}\vp_{\e}|<C$ uniformly.\\
\indent If
$Z=|\Delta^{\epsilon}_{e}u_{\e,\delta}|^2>|\Delta^{\epsilon}_{e}\vp_{\e}|^2$
at an interior point $(x_0,t_0)$, we can get a contradiction.
Therefore $Z\leq |\Delta^{\epsilon}_{e}\vp_{\e}|^2$ in the interior of $Q_T$. On the other hand,
$u_{\e}>\vp$ and then $\beta_{\delta}(u_{\e,\delta}-\vp_{\e})=0$ on
a uniform neighborhood of $\partial_l Q_T$. From the $C^2$-estimate
of the solution for the heat equation, we have
$|\Delta^{\epsilon}_{e}u_{\e,\delta}|^2
<C\sup_{Q_T}|u_{\e,\delta}|<C\sup_{Q_T}|\vp|$ on $\partial_lQ_T$.
Hence, by the maximum principle, $Z\leq
C(\|\vp\|_{C^1(Q_T)}+\|g\|_{C^1(\Omega)})$.
\end{proof}
\marginpar{\tiny\color{red}
  }
\begin{corollary} we have

\item $|u_{\e}(x_1,t)-u_{\e}(x_2,t)|\leq C(|x_1-x_2|$ for a uniform constant $C$ when
       $x_1-x_2\in \e\Z^n$.
       \label{cor-d-lip}
\end{corollary}
\begin{lemm}[Regularity in Time]\label{lem-1}
$$\|D_tu_{\e}(x,t)\|\leq C.$$
\end{lemm}
\begin{proof}
Let $T=\sup_{(x,t)\in Q_T}|(u_{\e ,\delta})_t|^2$ and assume that
the maximum $T$ is achieved at $(x_1,t_1)$. Then we have , at
$(x_1,t_1)$,
$$H[|(u_{\e ,\delta})_t|^2]\leq 0$$
 By taking a time derivative in (\ref{eq-li-penalty}) with respect to time $t$, we have
$$-H[(u_{\e ,\delta})_t]+\beta_{\delta}'(\cdot)((u_{\e ,\delta})_t-(\vp_{\e})_t)=0.$$
Hence
\begin{equation*}
\begin{aligned}
-H[|(u_{\e ,\delta})_t|^2]&+2|\nabla (u_{\e ,\delta})_t|^2\\
&+2\beta_{\delta}'(\cdot)(|(u_{\e ,\delta})_t|^2-(u_{\e
,\delta})_t(\vp_{\e})_t)=0.
\end{aligned}
\end{equation*}
We also know $|(\vp_{\e})_t|<C$ uniformly.

If $T=|(u_{\e ,\delta})_t|^2>|(\vp_{\e})_t|^2$ at an interior point
$(x_1,t_1)$,  we can get a contradiction. Therefore $T\leq |(\vp_{\e})_{t}|^2<C$ for some $C>0$ in the interior of $Q_T$. On the other hand, $0=u_{\e,\delta}\geq \vp_{\e}$ on
$\partial_l Q_T$ and $u_{\e, \delta}(x,0)=g(x)\geq \vp(x,0)$ then
$\beta_{\delta}(u_{\e,\delta}-\vp_{\e})=0$ on a small neighborhood
of $\partial_p Q_T$. By the $C^2$-estimate of the solution of the
heat equation, we get the desired bound on the boundary and the lemma follows.
\end{proof}
\begin{lemm}\label{lem-3}
When $\alpha<\alpha_*$, $u_{\epsilon}$ satisfies
\begin{equation*}
\vp(x,t)-C\frac{\e^{\beta}}{a_{\e}^{\beta-2}}\leq u_{\e}(x,t)
\end{equation*}
for some $\beta>n$. In addition, there is  a Lipschitz function $u$,
such that
\begin{enumerate}
\item
\begin{equation*}
-C\frac{\e^{\beta}}{a_{\e}^{\beta-2}}\leq u_{\e}(x,t)- u\leq 0
\end{equation*}
, which implies the uniform convergence of $u_{\e}$ to $u$.
\item $u$ is a least super-solution of (\ref{eq-la-thick}) in Theorem (\ref{thm-i}).
\end{enumerate}
\label{lem-thick}
\end{lemm}
\begin{proof}
(1.)Since $u_{\epsilon} \geq \varphi$ in
$T_{a_{\epsilon}}\times(0,T]$, we show that the inequality can be
satisfied in $Q_{T,a_{\epsilon}}$. For a given $\delta_0>0$, let
\begin{equation*}
h_{\e}(x)=k\sup_{m\in \e\Z,\,\, x\in
\Omega_{\epsilon}}\left[\frac{\e^{\beta}}{|x-m|^{\beta-2}}-\frac{\e^{\beta}}{a_{\e}^{\beta-2}}\right].
\end{equation*}
Then $H[ h_{\e}]\geq c_0k$ for $\beta>n$ and a uniform constant
$c_0$. In addition, $0\geq h_{\e}>
-\frac{k\e^{\beta}}{\alpha_{\e}^{\beta-2}}$ in $\Omega_{\epsilon}$.
For any point $(x_{0},t_0)$ in $Q_{T,a_{\epsilon}}$, we choose a
number $\rho_0$ and large numbers $k,M>0$ such that
\begin{equation*}
\begin{aligned}
\overline{h}(x,t)=-\frac{c_{0}k}{4n}|x-x_{0}|^2&+\D
\vp(x_{0},t_0)\cdot
(x-x_{0})\\
&+\vp(x_{0},t_0)+h_{\e}(x)+M(t-t_0)<0\leq u_{\e}(x,t)
\end{aligned}
\end{equation*}
on $\partial B_{\rho_0}(x_{0})\times [\frac{1}{2}t_0,t_0]$ and
$B_{\rho_0}(x_0)\times \{\frac{1}{2}t_0\}$ and
\begin{equation*}
\overline{h}(x,t)<\varphi(x,t)\leq u_{\e}(x,t)
\end{equation*}
on $\{\partial T_{a_{\epsilon}}\cap B_{\rho_0}(x_{0})\}\times
[\frac{1}{2}t_0,t_0]$. By the choice of numbers, we get
\begin{equation*}
H[\overline{h}]=\frac{c_{0}k}{2}-M \geq 0.
\end{equation*}
Therefore $\overline{h}\leq u_{\e}$ in $\{B_{\rho_0}(x_0)\backslash T_{a_{\epsilon}}\}\times[\frac{1}{2}t_0,t_0]$, which gives us $\vp(x_{0},t_0)-C\frac{\e^{\beta}}{a_{\e}^{\beta-2}}\leq u_{\e}(x_{0},t_0)$. Moreover, the least super solution $v(x,t)$ above $\vp(x,t)$ is
greater than $u_{\e}$ which is the least super-solution for a
smaller obstacle.
Similarly, the lower bound of $u_{\e}$ above implies $v(x,t)<u_{\e}+C\frac{\e^{\beta}}{a_{\e}^{\beta-2}}$.\\

Since $u_{\e}$ is a super-solution, this implies (2).


\end{proof}

\begin{lemm}\label{lem-series-for-corollary-2-10----3}
\item Set $a_{\e}=(\frac{\epsilon a^*_{\e}}{2})^{1/2}$. Then
\begin{equation*}
\osc_{\partial B_{a_{\epsilon}}(m)\times
[t_0-a^2_{\epsilon},t_0]}u_{\e}=O(\e^{\gamma}),
\end{equation*}
for $m\in \e\Z^n\cap\supp\vp$ and for some $0<\gamma\leq 1$.
\label{lem-osc-p}
\end{lemm}
\begin{proof}
If we make a scale $v_{\e}(x,t)=u_{\e}(a_{\e}x+m,a_{\e}^2 t+t_0 )$,
we have a bounded  caloric function in a large domain
$\left\{B_{\e/2a_{\e}}(0)\backslash
B_{a_{\epsilon}^{\ast}/a_{\e}}(0)\right\}\times[0,\e^2/a_{\e}^2]$
such that $v_{\epsilon}(x,t)\geq
\varphi_{\epsilon}(a_{\epsilon}x+m,t)\geq
\varphi(m,t)-Ca_{\epsilon}$ on
$B_{a^{\ast}_{\epsilon}/a_{\epsilon}}(0)\times[0,\epsilon^2/a_{\epsilon}^2]$.
We may expect almost Louville theorem saying that the oscillation on
the uniformly bounded set is of order $o(\e^{\gamma})$. Let $w_{\e}$
be a caloric replacement of $v_{\e}$ in
$B_{\e/2a_{\e}}\times[0,\e^2/a_{\e}^2]$. Then
$\osc_{B_1\times[\epsilon^2/a_{\epsilon}^2-1,\epsilon^2/a_{\epsilon}^2]}w_{\e}=o(\e^{\gamma})$
by applying the oscillation lemma,\cite{LB}, of the caloric
functions inductively:
\begin{equation*}
\osc_{_{{B_R(x_0)\times[t_0-R^2,t_0]}}}w_{\epsilon}<\delta_0\osc_{_{{B_{4R}(x_0)\times[t_0-(4R)^2,t_0]}}}w_{\epsilon}
\end{equation*}
for some $0<\delta_0<1$ and $\gamma \approx
\log_{\epsilon}\left(\frac{a_{\epsilon}}{\epsilon}\right)^{-\log_4\delta_0}=\frac{\log_4\delta_0^{-1}}{n-2}$.
It is noticeable that $\delta_0=1-\frac{1}{C_1}$ for $C_1>0$ which
comes from the Harnack inequality,
\begin{equation*}
\sup_{_{B_{R/2}(x_0)\times[t_0-\frac{5R^2}{4},t_0-R^2]}}w\leq
C_1\inf_{_{B_R(x_0)\times[t_0-R^2,t_0]}}w
\end{equation*}
for a positive caloric function $w$. Then the error
$v=v_{\e}-w_{\e}$ is also a caloric in
$\big\{B_{\e/2a_{\e}}(0)\backslash
B_{a_{\epsilon}^{\ast}/a_{\e}}(0)\big\}\times[0,\e^2/a_{\e}^2]$ and
$v=0$ on $\big\{\partial
B_{\e/2a_{\e}}(0)\big\}\times[0,\e^2/a_{\e}^2]$ and $B_{\epsilon
\backslash a_{\epsilon}}(0)\times\{0\}$, we also have
\begin{equation*}
0\leq v\leq 2\sup_{Q_T}\vp
\end{equation*}
in $\big\{B_{\e/2a_{\e}}(0)\backslash
B_{a^{\ast}_{\epsilon}/a_{\e}}(0)\big\}\times[0,\e^2/a_{\e}^2]$.
Since the harmonic function can be considered a stationary caloric
function, we have
\begin{equation*}
0\leq v(x,t)\leq
(\sup_{\Omega}\vp)\frac{(a_{\e})^{n-2}}{r^{n-2}},\qquad r=|x|
\end{equation*}
which means $\osc_{\partial B_1\times[0,1]}v=O(\e^{n-1})$. Therefore
we know
\begin{equation*}
\osc_{_{{\partial
B_{a_{\e}}(m)\times[t_0-a^2_{\epsilon},t_0]}}}u_{\e}=\osc_{_{{\partial
B_1(0)\times[\left(\epsilon/a_{\epsilon}\right)^2-1,\left(\epsilon/a_{\epsilon}\right)^2]}}}v_{\e}=o(\e^{\gamma}).
\end{equation*}
\end{proof}

By the Lemma \ref{lem-series-for-corollary-2-10----1}, \ref{lem-1} and \ref{lem-series-for-corollary-2-10----3}, we get the following corollary.
\begin{corollary}\label{corollary-2-10-}
we have
\begin{equation*}
|u_{\e}(x,t_1)-u_{\e}(y,t_2)|\leq
C_1|x-y|+C_2|t_1-t_2|^{\frac{1}{2}}+o(\e^{\gamma})
\end{equation*}
 for $(x,t_1),(y,t_2)\in \big(\cup_{m\in\e\Z} \partial
B_{a_{\e}}(m)\cap\Omega\big)\times (0,T])$.
\end{corollary}

\begin{lemm}
\item Set  $a_{\e}=(\frac{\epsilon a^*_{\e}}{2})^{1/2}$. Then
\begin{equation*}
\osc_{\{B_{\e}(m)\backslash
B_{a_{\e}}(m)\}\times[t_0-a^2_{\epsilon},t_0]}u_{\e}=o(\eta(\epsilon))
\end{equation*}
for $m\in \e\Z^n\cap\supp\vp$ and for some function $\eta(\epsilon)$
satisfying
\begin{equation*}
\eta(\epsilon) \to 0 \qquad \textrm{as} \quad \epsilon \to 0.
\end{equation*}
\label{lem-osc}
\end{lemm}
\begin{proof}
The Lemma \ref{lem-osc-p} tells us that $u_{\e}$ is almost
constant on a set $\{\partial
B_{a_{\e}}(m)\}\times[t_0-a^2_{\epsilon},t_0]$ whose radius is
greater than a critical rate $a^*_{\e}$. Let
\begin{equation*}
\tilde{u}_{\epsilon}(x,t)=\sup_{\partial
B_{a_{\epsilon}}\times\{t\}}u_{\epsilon}.
\end{equation*}
for $(x,t)\in \{Q^{\epsilon}_m\cap\Omega\} \times(0,T]$. Then, by the corollary \ref{corollary-2-10-}, we have
\begin{equation*}
|\tilde{u}_{\e}(x,t_1)-\tilde{u}_{\e}(y,t_2)|\leq
C_1|x-y|+C_2|t_1-t_2|^{\frac{1}{2}}+o(\e^{\gamma})
\end{equation*}
and
\begin{equation*}
|\tilde{u}_{\epsilon}(z,t)-u_{\epsilon}(z,t)|\leq
C{\epsilon}^{\gamma}
\end{equation*}
for all $(x,t_1), (y,t_2) \in \Omega\times(0,T]$, $(z,t) \in
\{\cup_{m\in \epsilon \mathbb{Z}^n}\partial B_{a_{\epsilon}}(m)\cap
\Omega\}\times(0,T]$ and for some $C<\infty$. Therefore there is a
limit $\tilde{u}(x,t)$ of $\tilde{u}_{\epsilon}(x,t)$ such that
\begin{equation*}
\sup_{\{\cup_{m \in \epsilon \mathbb{Z}^n}\partial
B_{a_{\epsilon}}(m)\cap
\Omega\}\times[0,\infty)}|u_{\epsilon}(x,t)-\tilde{u}(x,t)|=o(\eta(\epsilon)).
\end{equation*}
for some function $\eta(\epsilon)$ which goes to zero as $\epsilon
\to 0$. This estimate says that the values of
$\big(\tilde{u}(x,t)-C\eta(\epsilon)\big)\chi_{_{T_{a_{\epsilon}}}}$
plays as an obstacle below $u_{\epsilon}$ with a slow decay rate,
$a_{\epsilon}>> a^{\ast}_{\epsilon}$, in Lemma \ref{lem-3}, which
will gives us the conclusion.
\end{proof}


\subsection{Homogenized Equations}\label{subsec-homogenized} In this section, we are going to find homogenized equation satisfied by the limit $u$ of $u_{\e}$
through viscosity methods.

\begin{lemm}\label{lem-6}
Let $a^{\ast}_{\epsilon}=\epsilon^{{\alpha}_{\ast}}$ for
${\alpha}_{\ast}=\frac{n}{n-2}$ for $n \geq 3$ and
${a}^{\ast}_{\epsilon}=e^{-\frac{1}{\epsilon^2}}$ for $n=2$. Then for
$c_0{a}_{\epsilon}^{\ast} \leq {a}_{\epsilon} \leq
C_0{a}_{\epsilon}^{\ast}$, $u$ is a viscosity solution of
\begin{equation}\label{eq-homogenized}
\begin{cases}
\La u +\kappa_{B_{r_0}}(\varphi-u)_+-u_t=0 \qquad \qquad \mbox{in
$Q_T$}\\
\qquad u=0 \qquad \qquad \qquad \quad \qquad \qquad \mbox{on $\partial_lQ_T$}\\
\qquad u=g(x) \qquad \qquad \qquad \qquad \quad \mbox{in $\Omega
\times \{t=0\}$}
\end{cases}
\end{equation}
where $\kappa_{B_{r_0}}$ is the capacity of $B_{r_0}$ if
$r_0=\lim_{\epsilon \to
0}\frac{{a}_{\epsilon}}{{a}^{\ast}_{\epsilon}}$ exists.
\end{lemm}
\begin{proof}
First, we are going to show that $u$ is a sub-solution. If not,
there is a quadratic polynomial
\begin{equation*}
P(x,t)=-d(t-t^0)+\frac{1}{2}a_{ij}(x_i-x_i^0)(x_j-x_j^0)+b_i(x_i-x_i^0)+c
\end{equation*}
touching $u$ from above at $(x^0,t^0)$ and
$$H[P]+\kappa (\varphi-P)_+< -\mu_0 <0.$$
In a small neighborhood of $(x^0,t^0)$,
$B_{\eta}(x^0)\times[t^0-\eta^2,t^0]$, there is another quadratic
polynomial $Q(x,t)$ such that
\begin{equation*}
\begin{cases}
H[P]<H[Q] \qquad \qquad \qquad \qquad \mbox{in
$B_{\eta}(x^0)\times[t^0-\eta^2,t^0]$}\\
Q(x^0,t^0)<P(x^0,t^0)-\delta_0 \\
Q(x,t)>P(x,t) \qquad \mbox{on $\partial B_{\eta}(x^0)\times
[t^0-\eta^2,t^0]$ and $B_{\eta}(x^0)\times \{t^0-\eta^2\}$}.
\end{cases}
\end{equation*}
In addition, we can choose an appropriate number $\epsilon_0>0$ so
that $Q$ satisfies
\begin{equation*}
H[Q]+\kappa(\vp(x^0,t^0)-u(x^0,t^0)+\e_0):=H[Q]+\kappa\xi_0<-\frac{\mu_0}{2}<0
\end{equation*}
and 
\begin{equation*}
|Q(x,t)- Q(x^0,t^0)|+|\vp(x,t)-\vp(x^0,t^0)|<\e_0
\end{equation*} 
in $B_{\eta}(x^0,t^0)\times[t^0-\eta^2,t^0]$. Let us consider
$$Q_{\e}(x,t)=Q(x,t)+w_{\e}(x)\xi_0.$$
Then we have
\begin{equation*}
H[ Q_{\e}(x,t)]<-\frac{\mu_0}{2}<0
\end{equation*}
and
\begin{equation*}
\begin{aligned}
Q_{\e}(x,t)&=Q(x,t)+(\vp(x^0,t^0)-u(x^0,t^0)+\e_0)\\
&>Q(x,t)+(\vp(x^0,t^0)-Q(x^0,t^0)+\e_0)\\
&>\vp(x,t)
\end{aligned}
\end{equation*}
on $\{T_{a_{\epsilon}}\cap B_{\eta}(x^0)\}\times[t_0-\eta^2,t_0]$. Hence, by the maximum principle, $Q_{\e}(x,t)\geq \vp_{\e}(x,t)$ in $B_{\eta}(x^0)\times[t^0-\eta^2,t^0]$.\\
\indent Now we define the function
\begin{equation*}
v_{\epsilon}=\begin{cases}
             \min(u_{\epsilon},Q_{\epsilon}) \qquad x\in B_{\eta}(x^0)\\
              u_{\epsilon}\qquad \qquad \qquad  x\in \Omega\bs  B_{\eta}(x^0).
           \end{cases}
\end{equation*}
Since $\left(\frac{\epsilon
    a^{\ast}_{\epsilon}}{2}\right)^{\frac{1}{2}}=o(\epsilon)$ as
$\epsilon \to 0$, by Lemma \ref{lem-osc}, $u_{\epsilon}$ converges
uniformly  to $u$ in $\Omega$. Hence, for sufficiently small $\epsilon>0$, $Q_{\e}> u_{\e}$ on $\partial B_{\eta}(x^0)\times[t_0-\eta^2,t_0]$. Thus the function $v_{\epsilon}$ is well-defined and will be a viscosity
super-solution of (\ref{eq-homogenized}). Since $u_{\e}$ is the
smallest viscosity super-solution of (\ref{eq-homogenized}),
\begin{equation*}
u_{\e} \leq v_{\e} \leq Q_{\e}.
\end{equation*}
Letting $\e \to 0$, we have $u(x^0,t^0) \leq
Q(x^0,t^0)<P(x^0,t^0)=u(x^0,t^0)$ which is a contradiction.
By an argument similar to the proof of Lemma 4.1 in \cite{CL}, we can show that $u$ is also a viscosity super-solution of
(\ref{eq-homogenized}).
\end{proof}
\begin{lemm}
\item When $a_{\e}=0(\e^{\alpha})$ for a $ \alpha>\alpha_*$, the limit $u$ is a viscosity
      solution of
     \begin{equation*}
     \begin{cases}
     H[u]=0\qquad \qquad \qquad \textrm{in $Q_T$}\\
      u(x,t)=0\qquad \qquad \qquad \textrm{on $\partial_l Q_T$}\\
      u(x,0)=g\qquad \qquad \qquad \textrm{on $\Omega\times\{0\}$}
     \end{cases}
     \end{equation*}
     \label{lem-ho-thin}
\end{lemm}
\begin{proof}
For $\e>0$, $H[ u_{\e}]\leq 0$. Hence the limit also satisfies
$H[u]\leq 0$ in a viscosity sense. In order to show $u$ is a
sub-solution in $Q_T$, let us assume that there is a point $(x_0,t_0)\in
Q_T$ such that $H[P](x_0,t_0)\leq -\delta_0<0$ for a quadratic
polynomial $P$ such that $(P-u)$ has a minimum value zero at
$(x_0,t_0)$. We are going to choose a small neighborhood of
$(x_0,t_0)$, $B_{\eta}(x_0)\times[t_0-\eta^2,t_0]$, and a quadratic
polynomial $Q(x,t)$ such that
\begin{equation*}
\begin{cases}
Q(x,t)>P(x,t)\quad \mbox{on $\partial B_{\eta}(x_0)\times[t_0-\eta^2,t_0]$ and $B_{\eta}(x_0)\times\{t_0-\eta^2\}$}\\
H[Q]>H[P]\qquad \,\,\mbox{in $B_{\eta}(x_0)\times[t_0-\eta^2,t_0]$}\\
Q(x_0,t_0)<P(x_0,t_0)-\delta_0
\end{cases}
\end{equation*}

Let $Q_{\e}=Q(x,t)+(w_{\e}-\min w_{\e})$. Then $H[ Q_{\e}]=0$ and
$Q_{\e}\geq \vp_{\e}$ since $1-\min w_{\e}\ra \infty$ as $\e\ra 0$.
Hence $\min(u_{\e},Q_{\e})$ is a super-solution of (\ref{eq-main}),
but $\min(u_{\e},Q_{\e})(x_0,t_0)<u_{\e}(x_0,t_0)$, which is a
contradiction against the choice of $u_{\e}$.
\end{proof}
\begin{lemm}
\item When $a_{\e}=O(\e^{\alpha})$ for $\alpha<\alpha_*$, the limit $u$ is a least viscosity super solution of
      \begin{equation*}
      \begin{cases}
      H[ u] \leq 0\qquad \qquad \qquad \textrm{in $ Q_T$}\\
      u=0\qquad \qquad \qquad \textrm{on $\partial_l Q_T$}\\
      u\geq \vp\qquad \qquad \qquad\textrm{in $D$}\\
      u(x,0)=g(x)
      \end{cases}
      \end{equation*}
\label{lem-ho-thick}
\end{lemm}
\begin{proof}
The proof is very similar to that of the lemma 4.3 in \cite{CL}.
Likewise we only need to show $u\geq \vp$. Let us assume there is a
point $(x_0,t_0)$ such that $u(x_0,t_0)<\vp(x_0,t_0)$. We are going
to construct a corrector with an oscillation of order 1, which is
impossible in case that the decay rate of $a_{\e}$ is slow, Lemma
\ref{lem-5} and \ref{lem-thick}. For small $\e>0$, we have
\begin{equation*}
|u_{\e}(x,t_0)-u(x,t_0)|<\frac{1}{4}|u(x_0,t_0)-\vp(x_0,t_0)|
\end{equation*}
on $\big(B_{\eta}(x_0)\cap\Omega_{a_{\e}}\big)\times\{t_0\}$. For a
sufficiently large constant $M_1$, we set
\begin{equation*}
u_{\e}+M_1 |x-x_{0}|>\vp(x,t_0)
\end{equation*}
on $\partial B_{\eta}(x_0)\times \{t_0\}$. Then we can set a
periodic function
\begin{equation*}
\overline{w}_{\e}=\min_{m\in\e\Z^{n}}\bigg[\big\{u_{\e}(x-m,t_0)+M_1|x-m-x_{0}|\big\}\chi_{_{B_{\eta}(x_0-m)}}+M_2\chi_{_{\mathbb{R}^n\backslash
B_{\eta}(x_0-m)}}\bigg]
\end{equation*}
for a sufficiently large constant $M_2>0$ and then it is a
super-solution such that $\max \overline{w}_{\e}-\min
\overline{w}_{\e} >\frac14 |u(x_{0},t_0)-\vp(x_{0},t_0)|>0$ for
small $\e>0$ on $B_{\eta}(x_0)\times\{t_0\}$ . Hence we can extend
periodically $\overline{w}_{\e}$ so that we have global periodic
super-solution. But $\overline{w}_{\e}$ will not go to $0$ as $\e\ra
0$, which is a contradiction against Lemma (\ref{lem-5}).
\end{proof}
{\bf Proof of Theorem \ref{thm-i}} (1.) Set
$D=(\cup_{\e<\e_{o}}\cup_{m\in\e\Z^{n}}T_{m}^{\sqrt{\frac{\epsilon
a^{\ast}_{\e}}{2}}})\cap \Omega$. For any $\delta>0$, there is
$\e_{o}>0$ such that  $|D|<\delta$.
Corollary \ref{cor-d-lip} shows the uniform convergence of $u_{\e}$ on $\Omega\bs D$.\\
(2-a),(2-b), and (2-c) come from Lemma \ref{lem-6} , Lemma
\ref{lem-ho-thick} and Lemma \ref{lem-ho-thin}.\qed

\section{Elliptic Eigenvalue Problem in Perforated Domain}
\label{sec-eigen} Before we deal with $\epsilon$-problem for the porous medium equation, we consider nonlinear eigen value problem,
which will describe the behavior of the solution for the porous medium
equation in a neighborhood of $\partial\Omega$. Let's consider the solution $\vp_{\epsilon}(x)$ of
\begin{equation*}\label{EVe}
\begin{cases}
\begin{aligned}
\La \vp_{\epsilon}+\vp^p_{\epsilon}&=0,\quad 0<p<1  \qquad
\qquad \qquad
\textrm{in}\,\, \Omega_{a^{\ast}_{\epsilon}}\\
\vp_{\epsilon}&>0 \qquad \qquad \qquad \qquad \qquad \quad
\textrm{in}\,\,
\Omega_{a^{\ast}_{\epsilon}} \\
\vp_{\epsilon}&=0 \qquad \qquad \qquad \qquad \qquad \textrm{on}
\,\, T_{a^{\ast}_{\epsilon}} \cup
\partial\Omega_{a^{\ast}_{\epsilon}}.
\end{aligned}
\end{cases}
\tag{{\bf $\textrm{EV}_{\epsilon}$}}
\end{equation*}
\subsection{Discrete Nondegeneracy}
We need to construct appropriate barrier functions to estimate the
discrete gradient of a solution $\vp_{\epsilon}$ of \eqref{EVe} on the boundary.
\begin{lemm}
\item For each unit direction $e_i$ and $x \in \partial \Omega$, set
\begin{equation*}
\Delta^{\epsilon}_{e_i}\vp_{\e}=\frac{\vp_{\e}(x+\e
e_i)-\vp_{\e}(x)}{\e}
\end{equation*}
and
\begin{equation*}
\|\Delta^{\epsilon}_{e}\vp_{\e}(x)\|=\sqrt{|\sum_{i}\Delta^{\epsilon}_{e_i}\vp_{\e}|^2}.
\end{equation*}
Then there exist suitable constants $c>0$ and $C<\infty$ such that
\begin{equation*}
c<\|\Delta^{\epsilon}_{e}\vp_{\e}(x)\|<C
\end{equation*}
uniformly.
\end{lemm}
\begin{proof}
Let $h^+$ be a solution of
\begin{equation*}
\begin{cases}
\begin{aligned}
\La h^+&=- M^p-1 \qquad \qquad \qquad \mbox{in $\Omega$} \\
h^+&=0 \qquad \qquad \qquad \qquad \quad \mbox{on $\partial \Omega$}
\end{aligned}
\end{cases}
\end{equation*}
with $M\geq\sup_{\Omega_{a_{\epsilon}}^{\ast}} \vp_{\epsilon}$. Then, we have
\begin{equation*}
\vp_{\epsilon} \leq h^+ \qquad \mbox{in $\Omega_{a_{\epsilon}}^{\ast}$}
\end{equation*}
by the maximum principle and
\begin{equation*}
 \sup_{\partial\Omega} |\nabla h^+|<C
\end{equation*}
by the standard elliptic regularity theory. Thus, for $x\in
\partial \Omega$,
\begin{equation*}
\|\Delta^{\epsilon}_e\vp_{\e}(x)\|\leq \|\nabla h^+(x)\|<C
\end{equation*}
when we extend $\vp_{\e}$ to zero in $\mathbb{R}^n\bs\Omega$. \\
\indent To get a lower bound, we first show that the limit
function $\vp$, of $\vp_{\e}$, is not identically zero. Let
\begin{equation*}
\lambda_{\e}=\min_{\tilde{\vp}_{\e} \in H_0^1(\Omega_{a^{\ast}_{\epsilon}}),
\|\tilde{\vp}_{\e}\|_{_{L^{^{p+1}}(\Omega_{a^{\ast}_{\epsilon}})}}=1}\|\nabla \tilde{\vp}_{\e}\|_{L^{^2}}.
\end{equation*}
For $0\leq\eta\in C_0^{\infty}(\Omega)$ and corrector $w_{\epsilon}$ given in Section \ref{sec-obstacle}, set $\theta(x)=\eta(x)(1-w_{\e}(x))$.
Then
\begin{equation*}
\begin{aligned}
\int_{\Omega_{a^{\ast}_{\e}}}|\nabla
\eta(1-w_{\e})|^2dx=&\int_{\Omega_{a^{\ast}_{\e}}}\nabla
\big[\eta(1-w_{\e})\big]\cdot\nabla
\big[\eta(1-w_{\e})\big]dx\\
=&\int_{\Omega_{a^{\ast}_{\e}}}(1-w_{\e})^2|\nabla \eta|^2-2\eta(1-w_{\e})\nabla \eta\cdot\nabla w_{\e}+\eta^2|\nabla w_{\e}|^2dx\\
\leq&2\int_{\Omega_{a^{\ast}_{\e}}}(1-w_{\e})^2|\nabla
\eta|^2+\eta^2|\nabla w_{\e}|^2dx.
\end{aligned}
\end{equation*}
Since $\int_{\Omega_{a^{\ast}_{\e}}}|\nabla w_{\e}|^2dx<C$ for some
$0<C<\infty$, we get
\begin{equation*}
\left(\int_{\Omega_{a^{\ast}_{\e}}}|\nabla
\eta(1-w_{\e})|^2dx\right)^{\frac{1}{2}}\leq C_1
\end{equation*}
for some constant $0<C_1<\infty$. On the other hand,
\begin{equation*}
\int_{\Omega_{a^{\ast}_{\e}}}|\eta(1-w_{\e})|^{p+1}dx=C_{2,\e}
\end{equation*}
for some constant $C_{2,\epsilon}$ depending on $\epsilon$. Since $(1-w_{\e})\rightharpoonup 1$ in $L^2(\Omega)$. we get
\begin{equation*}
\lim_{\e \to 0}C_{2,\e}=C_2<+\infty.
\end{equation*}
Thus
\begin{equation*}
\left(\int_{\Omega_{a^{\ast}_{\e}}}|\frac{1}{C_{2,\e}}\eta(1-w_{\e})|^{p+1}dx\right)^{\frac{1}{p+1}}=1
\end{equation*}
and
\begin{equation*}
\left(\int_{\Omega_{a^{\ast}_{\e}}}|\frac{1}{C_{2,\e}}\nabla[\eta(1-w_{\e})]|^2\right)^{\frac{1}{2}}\leq
C_1/C_{2,\e}.
\end{equation*}
Therefore we have
\begin{equation*}
\lambda_{\e}<\frac{2C_1}{C_2}<+\infty.
\end{equation*}
Then, the sequence $\{\epsilon\}$ has a subsequence which we still denote by $\{\epsilon\}$ such that
\begin{equation}\label{eq-aligned-tilde-vp-epsilon-to-tilde-vp}
\begin{aligned}
\tilde{\vp}_{\e} &\rightharpoonup \tilde{\vp} \qquad \qquad \mbox{in $H^1_0(\Omega)$} \\
\lambda_{\e}&\to \lambda.
\end{aligned}
\end{equation}
Since $H^1_0$ is compactly embedded in $L^{p+1}$, $\tilde{\vp}_{\e}\to \tilde{\vp}$
in $L^{p+1}(\Omega)$ implies $\|\tilde{\vp}\|_{L^{p+1}(\Omega)}=1$. Note that $\lambda
\neq 0$. Otherwise $\tilde{\vp}$ satisfies
\begin{equation*}
\begin{cases}
\begin{aligned}
\La \tilde{\vp}&=0 \qquad \qquad \mbox{in $\Omega$}\\
\tilde{\vp}=&0 \qquad \qquad \quad \mbox{on $\partial \Omega$}.
\end{aligned}
\end{cases}
\end{equation*}
Then $\tilde{\vp}=0$ which gives a contradiction since
$\|\tilde{\vp}\|_{L^{p+1}(\Omega)}=1$. For each $\lambda_{\epsilon}$, the function $\vp_{\epsilon}=\lambda_{\epsilon}^{\frac{1}{1-p}}\tilde{\vp}_{\epsilon}$ can be the solution of \eqref{EVe}. By \eqref{eq-aligned-tilde-vp-epsilon-to-tilde-vp}, 
\begin{equation*}
\vp_{\epsilon}=\lambda_{\epsilon}^{\frac{1}{1-p}}\tilde{\vp}_{\epsilon} \rightharpoonup \lambda^{\frac{1}{1-p}}\tilde{\vp}=\vp \qquad \mbox{in $H^1_0(\Omega)$}.
\end{equation*}
Since $\|\vp\|_{L^{p+1}(\Omega)}=\lambda^{\frac{1}{1-p}}\|\tilde{\vp}\|_{L^{p+1}(\Omega)}>0$, there is some constant
$\delta_0>0$ such that
\begin{equation*}
\vp \geq \delta_0>0 \quad \mbox{in $D \subset \Omega$} \qquad
\mbox{and} \qquad |D|\neq 0.
\end{equation*}
Now we consider the $\vp_{\e}\chi_{_{D}}=\overline{\vp}_{\e}$ and
denote by $\psi_{\e}$ the minimizer of
\begin{equation*}
\int_{\Omega_{a^{\ast}_{\e}}}|\nabla \psi_{\e}|^2dx
\end{equation*}
in $K_{\epsilon}=\{\psi_{\e}\in H_0^1(\Omega_{\e}),\,\, \psi_{\e}
\geq \overline{\vp}_{\e}\}$. Then, by the Theorem 3.21 in
\cite{JKO}, $\psi_{\e} \rightharpoonup \psi $ in $H^1_0(\Omega)$ and
\begin{equation*}
\begin{cases}
\begin{aligned}
\psi \geq \overline{\vp}&=\vp\chi_{_{D}}\geq \delta_0 \qquad \qquad
\mbox{in $D$}\\
\La \psi &-\kappa\psi=0 \qquad \qquad \mbox{in $\Omega\bs D$}
\end{aligned}
\end{cases}
\end{equation*}
for $\kappa=\mbox{cap$(B_1)$}$. Since $\psi_{\e}$ satisfies the
harmonic equation in $\Omega_{a^{\ast}_{\e}}\bs
\{\psi_{\e}=\overline{\vp}_{\e}\}$ with $\psi_{\e}=\vp_{\e}=0$ on
$\partial \Omega_{a^{\ast}_{\e}}$, we get $\vp_{\e} \geq \psi_{\e}$ in
$\Omega_{a^{\ast}_{\e}}\bs \{\psi_{\e}=\overline{\vp}_{\e}\}$ and
then in a neighborhood of $\partial\Omega$ by the maximum principle.
On the other hand, by the Hopf principle,
\begin{equation*}
\inf_{x\in\partial \Omega}|\nabla \psi(x)|>\delta_1>0.
\end{equation*}
Hence, there is a lower bound of
$\|\Delta^{\epsilon}_e\psi_{\e}\|(x)$ if $x\in \partial \Omega$,
which means
\begin{equation*}
\|\Delta^{\epsilon}_e\vp_{\e}\|\geq\|\Delta^{\epsilon}_e\psi_{\e}\|>\delta_1>0
\qquad \qquad \mbox{for $x\in \partial \Omega$}.
\end{equation*}
Therefore, $\|\Delta^{\epsilon}_e\vp_{\e}\|$ is bounded below by some constant and the lemma follows.
\end{proof}

\subsection{Discrete Gradient Estimate}
\begin{lemm}[Discrete Gradient Estimate]
For the solution $\vp_{\epsilon}$ of ($\textrm{EV}_{\e}$),
\begin{equation*}
\|\Delta^{\e}_e\vp_{\e}(x)\|^2\leq C
\end{equation*}
for all $x \in \overline{\Omega}$ when we extend $\vp_{\e}(x)$ to
zero in $\mathbb{R}^n\bs\overline{\Omega}$.
\end{lemm}
\begin{proof}
Let $G_{\Omega}$ and $G_{\Omega, a^{\ast}_{\e}}$ are the Green functions
of the Laplace equation in $\Omega$ and $\Omega_{a^{\ast}_{\e}}$,
respectively. We choose constant $\gamma$ such that
\begin{equation*}
B_{2\gamma}(y)\subset \Omega_{
a^{\ast}_{\epsilon}} \qquad (y \in \Omega_{
a^{\ast}_{\epsilon}})
\end{equation*}
and let the function $G_{\Omega,
a^{\ast}_{\e},\gamma}$ to be a solution of
\begin{equation*}
\begin{cases}
\begin{aligned}
\La G_{\Omega,
a^{\ast}_{\e},\gamma}&=0 \qquad \qquad \qquad \qquad \mbox{in $\Omega_{a^{\ast}_{\e}}\bs B_{\gamma}(y)$} \\
G_{\Omega, a^{\ast}_{\e},\gamma}&=0 \qquad \qquad \qquad \qquad
\mbox{on $\partial T_{a^{\ast}_{\e}} \cup \partial \Omega$}\\
G_{\Omega, a^{\ast}_{\e},\gamma}(x,y)&=G_{\Omega}(x,y) \qquad \qquad
\mbox{on $B_{\gamma}(y)$}.
\end{aligned}
\end{cases}
\end{equation*}
Then we get
\begin{equation*}
G_{\Omega,a^{\ast}_{\e},\gamma} \leq G_{\Omega} \qquad \qquad
\mbox{in $\Omega_{a^{\ast}_{\e}}\bs B_{\gamma}(y)$}.
\end{equation*}
Therefore, we obtain
\begin{equation}\label{eq-3-2-2}
|\Delta^{\e}_e G_{\Omega, a^{\ast}_{\e},\gamma}|\leq |\Delta^{\e}_e
G_{\Omega}| \qquad \qquad  \mbox{on $ \partial \Omega $}.
\end{equation}
To get the estimate on $\partial B_{\gamma}(y)$, consider
the difference
\begin{equation*}
G(x,y)=G_{\Omega}(x,y)-G_{\Omega,a^{\ast}_{\e},\gamma}(x,y).
\end{equation*}
Then $G(x,y)$ satisfies
\begin{equation*}
\begin{cases}
\begin{aligned}
\La G&=0 \qquad \qquad \qquad \quad \mbox{in
$\Omega_{a^{\ast}_{\e}}\bs
B_{\gamma}(y)$}\\
G&=0 \qquad \qquad \qquad \quad \mbox{on $\partial \Omega \cup
\partial
B_{\gamma}(y)$}\\
G(x,y)&=G_{\Omega}(x,y) \qquad \quad \mbox{on $\partial
T_{a^{\ast}_{\e}}$}.
\end{aligned}
\end{cases}
\end{equation*}
Note that $G_{\Omega}(x,y)$ has similar behaivour to $O(|x-y|^{2-n})$ as $|x-y| \to 0$. Thus,
\begin{equation*}
\max_{_{|x-y|\geq 2\gamma}}G_{\Omega}(x,y) < \frac{1}{2}
\min_{_{|x-y|=\gamma}}G_{\Omega}(x,y)
\end{equation*}
for a sufficiently small $\gamma>0$. Thus, there exists a constant $C>0$ such that 
\begin{equation*}
\begin{aligned}
&G(x,y) \leq \max_{_{|x-y|\geq 2\gamma}}G_{\Omega}(x,y)<\min_{_{|x-y|=\gamma}}G_{\Omega}(x,y)-\max_{_{|x-y|\geq 2\gamma}}G_{\Omega}(x,y)\\
&\qquad \qquad \leq C{\Gamma}_{\gamma}(x-y)=C\left(\frac{1}{{\gamma}^{n-2}}-\frac{1}{|x-y|^{n-2}}\right) \qquad \mbox{on $\partial T_{a_{\epsilon}}^{\ast}$}.
\end{aligned}
\end{equation*}
Thus
\begin{equation*}
G(x,y) \leq C{\Gamma}_{\gamma}(x-y) \qquad \mbox{in $\Omega_{a_{\epsilon}^{\ast}}\bs B_{\gamma}(y)$}.
\end{equation*}
Since ${\Gamma}_{\gamma}(x-y)=
G_{\Omega}(x,y)=0$ on $\partial
B_{\gamma}(y)$, we have
\begin{equation}\label{eq-3-2-3}
|\Delta^{\e}_e G_{\Omega, a^{\ast}_{\e},\gamma}|\leq |\Delta^{\e}_e
G_{\Omega}|+|\Delta^{\e}_e G|\leq |\Delta^{\e}_e
G_{\Omega}|+|\Delta^{\e}_e {\Gamma}_{\gamma}| \qquad \mbox{on $\partial B_{\gamma}(y)$.}
\end{equation}
To show the estimate at the interior points, we use the approximation method. As in \cite{Fr}, $G_{\Omega,a^{\ast}_{\e},\gamma}$ can be approximated
by the solutions, $G_{\Omega,a^{\ast}_{\e},\gamma,\delta}$, of the
following penalized equations,
\begin{equation}\label{eq-3-2-1}
\begin{aligned}
\La G_{\Omega,a^{\ast}_{\e},\gamma,\delta}+\beta_{\delta}(&-G_{\Omega,a^{\ast}_{\e},\gamma,\delta}+G_{\Omega}\cdot\xi(x) )=0 \qquad \quad \textrm{in}\,\, \Omega\bs B_{\gamma}(y)\\
G_{\Omega,a^{\ast}_{\e},\gamma,\delta}&=0 \qquad \qquad \qquad
\qquad \qquad \qquad \quad \textrm{on} \,\,
\partial\Omega \\
G_{\Omega, a^{\ast}_{\e},\gamma,\delta}(x,y)&=G_{\Omega}(x,y) \qquad
\qquad \qquad \qquad \quad \mbox{on $B_{\gamma}(y)$}
\end{aligned}
\end{equation}
where $\beta_{\delta}(s)$ satisfies
\begin{equation*}
\begin{aligned}
\beta_{\delta}'(s) \geq 0,\quad &\beta_{\delta}''(s) \leq 0, \quad
\beta_{\delta}(0)=-1,\\
\beta_{\delta}(s)=0 \quad \textrm{for}\,\, s>&\delta, \qquad
\lim_{\delta \to 0}\beta_{\delta}(s) \to -\infty \quad \textrm{for}
\,\, s<0
\end{aligned}
\end{equation*}
and a $\e$-periodic function $\xi(x)\in C^{\infty}$ satisfies
\begin{equation*}
\begin{aligned}
0 &\leq \xi \leq  1, \qquad \xi=0 \quad \mbox{in
$T_{a^{\ast}_{\e}}$}, \qquad \xi=1 \quad \mbox{in
$\mathbb{R}^n_{{\e}^{\frac{n-1}{n-2}}}$},\\
\La &\xi=0 \quad \mbox{in $\R^N\bs\left\{T_{a^{\ast}_{\e}}\cup\mathbb{R}^n_{{\e}^{\frac{n-1}{n-2}}}\right\}$},\qquad
\La \xi \leq 0, \qquad \Delta^{\e}_e\xi=0.
\end{aligned}
\end{equation*}
Similar to the proof of Lemma
\ref{subsec-u-epsilon}, we get
\begin{equation}\label{eq-aligned-penalized-eq-green-omega-a-ast-e-gamma-delta-43}
\begin{aligned}
\La (|\Delta^{\e}_e
G_{\Omega,a^{\ast}_{\e},\gamma,\delta}|^2)&-2|\nabla (\Delta^{\e}_e
G_{\Omega,a^{\ast}_{\e},\gamma,\delta})|^2\\
&-2\beta'_{\delta}(\cdot)\big(|\Delta^{\e}_e
G_{\Omega,a^{\ast}_{\e},\gamma,\delta}|^2-\Delta^{\e}_e
G_{\Omega,a^{\ast}_{\e},\gamma,\delta}\cdot\xi\Delta^{\e}_e
G_{\Omega} \big)=0.
\end{aligned}
\end{equation}
Suppose that $\max_{\Omega\bs B_{\gamma}(y)}|\Delta^{\e}_e
G_{\Omega,a^{\ast}_{\e},\gamma,\delta}|^2$ occurs at an interior point $x_0$. If $|\Delta^{\e}_e
G_{\Omega,a^{\ast}_{\e},\gamma,\delta}|^2(x_0)>|\Delta^{\e}_e
G_{\Omega}|^2(x_0)$, then by \eqref{eq-aligned-penalized-eq-green-omega-a-ast-e-gamma-delta-43} we get
\begin{equation*}
\begin{aligned}
0>\La (|\Delta^{\e}_e
G_{\Omega,a^{\ast}_{\e},\gamma,\delta}|^2)&-2|\nabla (\Delta^{\e}_e
G_{\Omega,a^{\ast}_{\e},\gamma,\delta})|^2\\
&-2\beta'_{\delta}(\cdot)\big(|\Delta^{\e}_e
G_{\Omega,a^{\ast}_{\e},\gamma,\delta}|^2-\Delta^{\e}_e
G_{\Omega,a^{\ast}_{\e},\gamma,\delta}\cdot\xi\Delta^{\e}_e
G_{\Omega} \big)=0.
\end{aligned}
\end{equation*}
Therefore 
\begin{equation}\label{eq-3-2-3-3-2-3-3-2-3}
|\Delta^{\e}_e
G_{\Omega,a^{\ast}_{\e},\gamma,\delta}|^2\leq |\Delta^{\e}_e
G_{\Omega}|^2 \qquad \mbox{in the interior of $\Omega\bs B_{\gamma}(y)$.}
\end{equation} 
By \eqref{eq-3-2-2}, \eqref{eq-3-2-3} and \eqref{eq-3-2-3-3-2-3-3-2-3}
\begin{equation*}
|\Delta^{\e}_e G_{\Omega,a^{\ast}_{\e},\gamma,\delta}|\leq
|\Delta^{\e}_e
G_{\Omega}|+|\Delta^{\e}_e {\Gamma}_{\gamma}| \qquad \mbox{in $\Omega\bs
B_{\gamma}(y)$}.
\end{equation*}
By taking $\delta \to 0$ and $\gamma \to 0$, we obtain
\begin{equation*}
|\Delta^{\e}_e G_{\Omega,a^{\ast}_{\e}}|\leq |\Delta^{\e}_e
G_{\Omega}|+|\Delta^{\e}_e {\Gamma}_{\gamma}|\qquad \mbox{in $\Omega_{a^{\ast}_{\e}}$}.
\end{equation*}
Since
\begin{equation*}
\vp_{\epsilon}(x)=\int_{\Omega_{a^{\ast}_{\epsilon}}}G_{\Omega,a^{\ast}_{\epsilon}}(x,y)\vp_{\epsilon}^p(y)
dy,
\end{equation*}
we get
\begin{equation*}
\begin{aligned}
|\Delta_{e}^{\epsilon}\vp_{\epsilon}(x)|& \leq
\left|\int_{\Omega_{a^{\ast}_{\epsilon}}}\Delta_{e}^{\epsilon}G_{\Omega,a^{\ast}_{\e}}(x,y)\cdot
\vp_{\epsilon}^p(y)dy\right| \\
&\leq
\int_{\Omega_{a^{\ast}_{\epsilon}}}\left(|\Delta_{e}^{\epsilon}G_{\Omega}(x,y)|+|\Delta^{\e}_e {\Gamma}_{\gamma}|\right)\vp_{\epsilon}^p(y)dy.
\end{aligned}
\end{equation*}
Since $\lim_{\e \to 0}|\Delta^{\e}_e G_{\Omega}|=|\nabla_e
G_{\Omega}|$ and $|\nabla G_{\Omega}(x,y)|\approx O(|x-y|^{1-n})$ as
$|x-y|\to 0$, we get, sufficiently small $\e>0$,
\begin{equation*}
|\Delta_{e}^{\epsilon}\vp_{\epsilon}(x)| \leq
C\int_{\Omega_{a^{\ast}_{\epsilon}}}(|x-y|^{1-n}+1)\vp_{\epsilon}^p(y)dy<\infty
\end{equation*}
and lemma follows.
\end{proof}
\subsection{Almost Flatness}\label{sec-1}

\begin{lemm}\label{lem-2}
Set $a_{\epsilon}=(\frac{\epsilon a^{\ast}_{\epsilon}}{2})^{1/2}$.
Then
\begin{equation*}
\osc_{\partial
B_{a_{\epsilon}}(m)}\vp_{\epsilon}=o({\epsilon}^{\gamma})
\end{equation*}
for $m \in \epsilon \mathbb{Z}^n\cap \Omega$ and for some
$0<\gamma<1$.
\end{lemm}
\begin{proof}
If we consider the scaled function
$v_{\epsilon}(x)=\vp_{\epsilon}(a_{\epsilon}x+m)$, $v_{\e}$ will be
bounded in $B_{\epsilon/2a_{\epsilon}} \backslash
B_{a^{\ast}_{\epsilon}/a_{\epsilon}}$ and $v_{\epsilon}=0$ on
$\partial B_{a^{\ast}_{\epsilon}/a_{\epsilon}}$. $v_{\epsilon}$ also
satisfies
\begin{equation*}
\La v_{\epsilon}=-a^2_{\epsilon}v^p_{\epsilon}.
\end{equation*}
Let $g_{\epsilon}$ be a harmonic replacement of $v_{\epsilon}$ in
$B_{\epsilon/2a_{\epsilon}} \backslash
B_{a^{\ast}_{\epsilon}/a_{\epsilon}}$. Following the proof of Lemma
3.5 in \cite{CL} which is similar to Lemma \ref{lem-osc-p}, we
get $\osc_{\partial B_1}g_{\epsilon}=o(\epsilon^{\gamma})$. Let's consider
\begin{equation*}
h(r)=\frac{M^p
a^2_{\epsilon}}{2n}\left(\frac{\epsilon^2}{4a^2_{\epsilon}}-r^2\right)
\end{equation*}
with $r=|x|$ and $M=\sup v_{\epsilon}$. Then
\begin{equation*}
\begin{aligned}
\La h&=-a^2_{\epsilon}M^p\leq -a^2_{\epsilon}v^p_{\epsilon}=\La
(v_{\epsilon}-g_{\epsilon}) \qquad \mbox{in
$B_{\epsilon/2a_{\epsilon}}
\backslash B_{a^{\ast}_{\epsilon}/a_{\epsilon}}$}\\
h &\geq 0=v_{\epsilon}-g_{\epsilon} \qquad \qquad \qquad \qquad
\qquad \mbox{on $\partial \left\{B_{\epsilon/2a_{\epsilon}}
\backslash B_{a^{\ast}_{\epsilon}/a_{\epsilon}}\right\}$}.
\end{aligned}
\end{equation*}
By the maximum principle, we get
\begin{equation*}
g_{\epsilon}\leq v_{\epsilon}\leq g_{\epsilon}+h \leq
g_{\epsilon}+C\epsilon^2
\end{equation*}
for some $C>0$ on $\partial B_1$. Thus
\begin{equation*}
\osc_{\partial B_1}v_{\e}\leq \osc_{\partial B_1}g_{\e}+C\e^2 \leq
o(\e^{\gamma}).
\end{equation*}
 If we rescale $v_{\e}$ back to $\vp_{\e}$, we can get the desired conclusion.
\end{proof}
\begin{lemm}
Set $a_{\epsilon}=(\frac{\epsilon a^{\ast}_{\epsilon}}{2})^{1/2}$.
Then
\begin{equation*}
\osc_{B_{\frac{\epsilon}{2}}(m)\backslash
B_{a_{\epsilon}}(m)}\vp_{\epsilon}=o({\epsilon}^{\tilde{\gamma}})
\end{equation*}
for $m \in \epsilon \mathbb{Z}^n\cap \Omega$ and for some
$0<\tilde{\gamma}<1$.
\end{lemm}
\begin{proof}
By the Section 2.1 in \cite{CL}, there is a periodic corrector
$w_{\epsilon}$ having properties
\begin{equation*}
\La w_{\epsilon}=k \quad \mbox{and} \quad |1-w_{\epsilon}|\leq
C{\epsilon}^{2-\frac{\beta-2}{n-2}} \qquad \mbox{in
$\mathbb{R}^n_{a_{\epsilon}}$}
\end{equation*}
for $k>0$ and $n<\beta<2(n-1)$. Let $L', N'>0$ be the constants to be determined later. We define
the barrier function
\begin{equation*}
\tilde{w}_{\epsilon}(x)=\big[1-w_{\epsilon}(x)\big]+L'|x-m|^2+M'+N'\epsilon
\end{equation*}
with
\begin{equation*}
M'=\sup_{\partial B_{a_{\epsilon}}(m)}\vp_{\epsilon}.
\end{equation*}
Then we can select sufficiently large numbers $k, L'\gg 1$ and $N'$
so that $\tilde{w}_{\epsilon}$ satisfies
\begin{equation*}
\begin{aligned}
\La \tilde{w}_{\epsilon}&\leq\La \vp_{\epsilon}\qquad \qquad \qquad \mbox{in $\Omega_{a_{\epsilon}}$}\\
\tilde{w}_{\epsilon}&\geq \vp_{\epsilon} \qquad \qquad \qquad \,\,
\mbox{on $\partial \Omega_{a_{\epsilon}}$}.
\end{aligned}
\end{equation*}
By the comparison principle, we get
\begin{equation*}
\vp_{\epsilon} \leq \tilde{w}_{\epsilon} \qquad \qquad \mbox{in
$\Omega_{a_{\epsilon}}$}
\end{equation*}
Similarly, we get
\begin{equation*}
\overline{w}_{\epsilon} \leq \vp_{\epsilon} \qquad \qquad \mbox{in
$\Omega_{a_{\epsilon}}$}
\end{equation*}
where
\begin{equation*}
\overline{w}_{\epsilon}=\big[w_{\epsilon}(x)-1\big]-l'|x-m|^2+m'-n'\epsilon, \qquad m'=\inf_{\partial B_{a_{\epsilon}}(m)}\vp_{\epsilon}
\end{equation*}
for sufficiently large numbers $l', n'>0$. Since $B_{\frac{\epsilon}{2}}(m)\backslash B_{a_{\epsilon}}(m)$ is a
small region, we get
\begin{equation*}
|\tilde{w}_{\epsilon}-\overline{w}_{\epsilon}|\leq
o(\epsilon^{\overline{\gamma}}) \qquad \qquad \mbox{in
$B_{\frac{\epsilon}{2}}(m)\backslash B_{a_{\epsilon}}(m)$}
\end{equation*}
with $0<\overline{\gamma}<1$ and lemma follows.
\end{proof}
\subsection{Correctibility Condition I}
Likewise highly oscillating obstacle problems, we need an
appropriate corrector. However, unlike the highly oscillating obstacle
problem, the corrector $\tilde{w}_{\e}$ should be a super-harmonic
with
\begin{equation*}
\tilde{w}_{\e}=0 \qquad \qquad \mbox{in $T_{a^{\ast}_{\e}}$}.
\end{equation*}
Since the solution $w_{\epsilon}$ of \eqref{eq-cases-definition-of-corrector-cite-by-CL-5} is sub-harmonic with $w_{\epsilon}=1$ in $T_{a^{\ast}_{\e}}$, it is natural to consider the function
\begin{equation*}
\tilde{w}_{\e}=b-bw_{\epsilon}
\end{equation*}
for some constant $b>0$.
\begin{lemm}\label{lemma-correctibility-condition-1-0000000000}
\item  Let $k_{b,\epsilon}$ be such that
\begin{equation*}
\La (b-bw_{\epsilon})+(b-bw_{\epsilon})^p=-b\La
w_{\epsilon}+(b-bw_{\epsilon})^p=k_{b,\epsilon}.
\end{equation*}
Then we have
\begin{equation*}
-b\kappa_{_{B_{r_0}}}=k_b-b^p
\end{equation*}
where $k_{b}=\lim_{\epsilon\to 0}k_{b,\epsilon}$ and $\kappa_{_{B_{r_0}}}$ is the harmonic capacity of $B_{r_0}$.
\end{lemm}

\begin{proof}
Set $v_{\epsilon}(x)=w_{\epsilon}(a^{\ast}_{\epsilon}x+m)$, then
$v_{\epsilon}$ satisfies
\begin{equation*}
\begin{cases}
-b\La
v_{\epsilon}+(a^{\ast}_{\epsilon})^2(b-bv_{\epsilon})^p=k_{b,\epsilon}(a^{\ast}_{\epsilon})^2
\qquad \qquad \mbox{in
$Q^{^{\frac{\epsilon}{a^{\ast}_{\epsilon}}}}_0\backslash B_1$}\\
v_{\epsilon}=1 \qquad \qquad \qquad \qquad \qquad \qquad \qquad
\qquad \mbox{on $\partial
B_1$}\\
v_{\epsilon}=|\nu\cdot\nabla v_{\epsilon}|=0 \qquad \qquad \qquad \qquad
\qquad \qquad \mbox{on $\partial
Q^{^{\frac{\epsilon}{a^{\ast}_{\epsilon}}}}_0$}.
\end{cases}
\end{equation*}

Thus, we have
\begin{equation*}
-b\int_{_{Q^{^{\frac{\epsilon}{a^{\ast}_{\epsilon}}}}_0\backslash
B_1}}\La
v_{\epsilon}dx=(a^{\ast}_{\epsilon})^2\int_{_{Q^{^{\frac{\epsilon}{a^{\ast}_{\epsilon}}}}_0\backslash
B_1}}k_{b,\epsilon}-b^p(1-v_{\epsilon})^pdx .
\end{equation*}
On the other hand, from the elliptic uniform estimates, \cite{GT}, $v_{\e}\ra v$ converges to a 
potential function $v$ of $B_1$ in $C^2$-norm on any bounded set  where  $\La u=0$ on $\R^n\bs B_1$, $v=1$ on
$\partial B_1$, and $v\ra 0$ as $|x|\ra \infty$.  Then  we get
\begin{equation*}
\begin{split}
-b&\int_{_{Q^{^{\frac{\epsilon}{a^{\ast}_{\epsilon}}}}_0\backslash
B_1}}\La v_{\epsilon}dx=-b\int_{_{\partial
\left\{Q^{^{\frac{\epsilon}{a^{\ast}_{\epsilon}}}}_0\backslash
B_1\right\}}}\nabla v_{\e}\cdot\nu
d\sigma_{x}\\
&  =-b\int_{_{\partial B_1}}\nabla v_{\e}\cdot (-\nu) d\sigma_{x}\ra -b\int_{_{\partial B_1}}\nabla v\cdot (-\nu) d\sigma_{x}=-b\kappa_{_{B_1}}.
\end{split}
\end{equation*}
as $\epsilon$ goes to zero. And  we also have
\begin{equation}\label{eq-3-4-1}
-b\kappa_{_{B_1}}=\lim_{\e \to
0}\big[(k_{b,\e}-b^p)(a^{\ast}_{\e})^2\cdot(\frac{\e}{a^{\ast}_{\e}})^n\big]
=\frac{1}{r^{n-2}_0}(k_b-b^p)
\end{equation}
where $\kappa_{_{B_{1}}}$ is the harmonic capacity of $B_{1}$ and $k_{b}=\lim_{\epsilon\to 0}k_{b,\epsilon}$. If we
multiply equation (\ref{eq-3-4-1}) by $r^{n-2}_0$, we obtain
\begin{equation*}
-b\kappa_{_{B_{r_0}}}=k_b-b^p
\end{equation*}
where $\kappa_{_{B_{r_0}}}$ is the harmonic capacity of $B_{r_0}$.
\end{proof}
\subsection{Homogenized Equation}
\begin{theo}\label{thm-eigen}
\item \begin{enumerate}
\item(The Concept of Convergence)\\
There is a continuous function $\vp$ such that $\vp_{\epsilon}\to
\vp$ in $\Omega$ with respect to $L^q$-norm, for $q>0$. and for any
$\delta>0$, there is a subset $D_{\delta} \subset \Omega$ and
$\epsilon_0$ such that, for $0<\epsilon<\epsilon_0$,
$\vp_{\epsilon}\to \vp$ uniformly in $D_{\delta}$ as $\epsilon \to
0$
and $|\Omega \backslash D_{\delta}|<\delta$.\\
\item Let $a^{\ast}_{\epsilon}=\epsilon^{{\alpha}_{\ast}}$ for
${\alpha}_{\ast}=\frac{n}{n-2}$ for $n \geq 3$ and
$a^{\ast}_{\epsilon}=e^{-\frac{1}{\epsilon^2}}$ for $n=2$. Then for
$c_0a_{\epsilon}^{\ast} \leq a_{\epsilon} \leq
C_0a_{\epsilon}^{\ast}$, $u$ is a viscosity solution of
\begin{equation}\label{eq-in-thm-main-eigenvalue-homogenized-equation-345}
\begin{cases}
\La \vp -\kappa_{B_{r_0}}\vp+\vp^p=0 \qquad \qquad \quad 
\mbox{in
$\Omega$}\\
\qquad \vp=0 \qquad \qquad \qquad \quad \qquad \quad \mbox{on $\partial \Omega$}\\
\qquad \vp>0 \qquad \qquad \qquad \qquad \qquad \mbox{in $\Omega$}
\end{cases}
\end{equation}
where $\kappa_{B_{r_0}}$ is the capacity of $B_{r_0}$ if
$r_0=\lim_{\epsilon \to
0}\frac{{a}_{\epsilon}}{{a}^{\ast}_{\epsilon}}$ exists.
\end{enumerate}
\end{theo}

\begin{proof}
By an argument similar to the proof of Theorem \ref{theo-main-heat-operator-71930}, (1) holds. \\
\indent (2). For $\epsilon>0$,
\begin{equation*}
\La \vp_{\epsilon}-\kappa_{_{B_{r_0}}}
\vp_{\epsilon}+\vp^p_{\epsilon}=-\kappa_{_{B_{r_0}}}
\vp_{\epsilon}\leq 0.
\end{equation*}
Hence, the limit also satisfies
\begin{equation*}
\La \vp-\kappa_{_{B_{r_0}}} \vp+\vp^p\leq 0
\end{equation*}
in a viscosity sense. Thus, we are going to show that $\vp$ is a
sub-solution. Let us assume that there is a parabola $P$ touching $u$
from above at $x_0$ and
\begin{equation*}
\La P-\kappa_{_{B_{r_0}}} P +P^p\leq -2\delta_0<0.
\end{equation*}
In a small neighborhood of $x_0$, $B_{\eta}(x_0)$, there is another
parabola $Q$ such that
\begin{equation*}
\begin{cases}
D^2Q > D^2P \qquad \qquad \qquad \mbox{in $B_{\eta}(x_0)$}\\
Q(x_0)<P(x_0)-\delta \\
Q(x)>P(x) \qquad \qquad \qquad \textrm{on}\,\, \partial
B_{\eta}(x_0).
\end{cases}
\end{equation*}
In addition, for $Q(x_1)=\min_{B_{\eta}(x_0)}Q(x)$,
\begin{equation*}
\La Q-\kappa_{_{B_{r_0}}} Q(x_1)+Q(x_1)^p \leq -\delta_0<0 \quad
\textrm{and} \quad |Q(x)-Q(x_1)|<C\eta
\end{equation*}
in $B_{\eta}(x_0)$. Then the function
\begin{equation*}
Q_{\epsilon}(x)=Q(x)-w_{\epsilon}(x)Q(x_1)
\end{equation*}
satisfies
\begin{equation*}
\La Q_{\epsilon}+Q^p_{\epsilon}\leq\La Q-Q(x_1)\La
w_{\epsilon}+(1-w_{\epsilon})^pQ(x_1)^p+(C\eta)^p
\end{equation*}
in $B_{\eta}(x_0)\cap \Omega_{a_{\epsilon}}$. By correctibility
condition I, Lemma \ref{lemma-correctibility-condition-1-0000000000}, 
\begin{equation*}
\begin{aligned}
\La Q_{\epsilon}+Q^p_{\epsilon}&\leq \La Q+k_{Q(x_1),\epsilon}+(C\eta)^p\\
&\leq \La Q+k_{Q(x_1)}+\frac{\delta_0}{4}+(C\eta)^p\\
&\leq \La Q-\kappa_{_{B_{r_0}}}
Q(x_1)+Q(x_1)^p+\frac{\delta_0}{2}\leq -\frac{\delta_0}{2}<0
\end{aligned}
\end{equation*}
for small $\epsilon, \eta>0$. So $\La Q_{\epsilon}+Q^p_{\epsilon}<0$
and $Q_{\epsilon}\geq u_{\epsilon}$ on $\partial \{B_{\rho}(x_0)\cap
\Omega_{a_{\epsilon}}\}$ for some $\rho>0$. By a comparison principle
, we get
\begin{equation*}
u_{\epsilon}\leq Q_{\epsilon}
\end{equation*}
in $B_{\rho}(x_0)\cap \Omega_{a_{\epsilon}}$. Thus we get
$Q_{\epsilon}(x_0)\geq \vp_{\epsilon}(x_0)$ and then $Q(x_0) \geq
\vp(x_0)$. On the other hand, $Q(x_0)<P(x_0)-\delta<\vp(x_0)$, which
is a contradiction. Therefore $\vp$ is a viscosity solution of \eqref{eq-in-thm-main-eigenvalue-homogenized-equation-345}.
\end{proof}
\section{Porous Medium Equations in a fixed Perforated Domain}\label{sec-por}
Now we can consider the following porous medium equations. The main question is to find the
viscosity solution  $u_{\e}(x,t)$ s.t.
\begin{equation}
\begin{cases}
\La u_{\epsilon}^m -\partial_tu_{\epsilon}=0 \qquad \qquad \mbox{in $Q_{T,a^{\ast}_{\epsilon}}\,\,(=\Omega_{a^{\ast}_{\epsilon}}\times(0,T])$}\\
u_{\epsilon} =0 \qquad \qquad \qquad \qquad \mbox{on $\partial_lQ_{T,a^{\ast}_{\epsilon}}\,\,(=\partial \Omega_{a^{\ast}_{\epsilon}}\times(0,T])$}\\
u_{\e}=g_{\epsilon} \qquad \qquad \qquad \qquad \mbox{on
$\Omega_{a^{\ast}_{\epsilon}}\times\{0\}$}
\end{cases}
\tag{{\bf $PME_{\e}^1$}} \label{eq-main1}
\end{equation}
where $1<m<\infty$ and $g_{\e}(x)= g(x)\xi(x)$ for a smooth function
$g(x) \in C^{\infty}_0(\Omega)$ satisfying
\begin{equation*}
0<\delta_0<|\nabla g|<C \qquad \qquad \mbox{on $\partial \Omega$}
\end{equation*}
and a $\e$-periodic function $\xi(x)\in C^{\infty}$ satisfying
\begin{equation}\label{eq-aligned-epsilon-periodic-in-PME-000}
\begin{aligned}
0 &\leq \xi \leq  1, \qquad \xi=0 \quad \mbox{in
$T_{a^{\ast}_{\e}}$}, \qquad \xi=1 \quad \mbox{in
$\mathbb{R}^n_{{\e}^{\frac{n-1}{n-2}}}$},\\
\La &\xi=0 \quad \mbox{in $\R^N\bs\left\{T_{a^{\ast}_{\e}}\cup\mathbb{R}^n_{{\e}^{\frac{n-1}{n-2}}}\right\}$},\qquad
\La \xi \leq 0, \qquad \Delta^{\e}_e\xi=0.
\end{aligned}
\end{equation}
Set $v_{\epsilon}=u^m_{\epsilon}$ which is a flux. Then $v_{\epsilon}$ satisfies
\begin{equation}
\begin{cases}
v^{1-\frac{1}{m}}_{\epsilon}\La v_{\epsilon} -\partial_tv_{\epsilon}=0 \qquad \qquad \mbox{in $Q_{T_{a^{\ast}_{\epsilon}}}$}\\
v_{\epsilon} =0 \qquad \qquad \qquad \qquad \quad \mbox{on $\partial_l Q_{T_{a^{\ast}_{\epsilon}}}$}\\
v_{\e}=g^m_{\e} \qquad \qquad \qquad \qquad \mbox{on
$\Omega_{a^{\ast}_{\epsilon}}\times\{0\}$}.
\end{cases}
\tag{{\bf $PME_{\e}^2$}} \label{eq-main2}
\end{equation}
In this section, we deal with the properties and homogenization for the solution $v$.
\subsection{Discrete Nondegeneracy}
Let $\varphi_{\epsilon}$ be a solution of the boundary value problem
\begin{equation*}
\begin{cases}
\begin{aligned}
\La \varphi_{\epsilon}&+\varphi_{\epsilon}^{\frac{1}{m}}=0 \qquad \qquad \qquad \mbox{in $\Omega_{a^{\ast}_{\epsilon}}$}\\
\varphi_{\epsilon}&=0 \qquad \qquad \qquad \qquad \mbox{on $\partial
\Omega_{a^{\ast}_{\epsilon}}$}.
\end{aligned}
\end{cases}
\end{equation*}
It is easy to see that the function
\begin{equation*}
V_{\epsilon,\lambda}(x,t)=\frac{\alpha\varphi_{\epsilon}(x)}{(\lambda+t)^{\frac{m}{m-1}}}, \qquad \qquad \alpha=\left(\frac{m}{m-1}\right)^{\frac{m}{m-1}}
\end{equation*}
satisfies the equation
\begin{equation*}
 V_{\epsilon,\lambda}^{1-\frac{1}{m}}\La
V_{\epsilon,\lambda}- (V_{\epsilon,\lambda})_t=0.
\end{equation*} 
As in Lemma \ref{lem-2}, the rescaled function $\vp_{\epsilon}(a_{\epsilon}x+m)$ approach the harmonic function in $B_{\frac{\epsilon}{2a_{\epsilon}}}\bs B_{\frac{a^{\ast}_{\epsilon}}{a_{\epsilon}}}$ as $\epsilon \to 0$. Hence, for sufficiently small $\epsilon>0$, $\vp_{\epsilon}$ becomes almost harmonic near $T_{a^{\ast}_{\epsilon}}$. Thus, $\vp_{\epsilon}$ is equvalent to the $\e$-periodic function $\xi$ in \eqref{eq-aligned-epsilon-periodic-in-PME-000} near $T_{a^{\ast}_{\epsilon}}$, i.e., there exist some constants $0<c\leq C<\infty$ such that
\begin{equation*}
c\vp_{\epsilon}\leq \xi \leq C\vp_{\epsilon} \qquad \mbox{near $T_{a^{\ast}_{\epsilon}}$}.
\end{equation*}
Therefore,  we can take constants
$0<\lambda_2\leq \lambda_1<\infty$ such that
\begin{equation*}
V_{\epsilon,\lambda_1} \leq v_{\epsilon} \leq V_{\epsilon,\lambda_2}
\end{equation*}
for the solution $v_{\epsilon}$ of the initial value problem
(\ref{eq-main2}). Therefore, by the nondegeneracy of
$\varphi_{\epsilon}$ in a neighborhood of $\partial\Omega$, we can get the
following result.
\begin{lemm}\label{lem-4}
\item For each unit direction $e$ and $x \in \partial \Omega$, set
\begin{equation*}
\Delta^{\epsilon}_{e}v_{\e}=\frac{v_{\e}(x+\e e,t)-v_{\e}(x,t)}{\e}.
\end{equation*}
Then there exist suitable constants $c>0$ and $C<\infty$ such that
\begin{equation*}
c<|\Delta^{\epsilon}_{e}v_{\e}(x,t)|<C
\end{equation*}
uniformly.
\end{lemm}
\subsection{Almost Flatness}
For small $\delta_0>0$, we consider the set
\begin{equation*}
T_{\tilde{a}^{\ast}_{\epsilon}}=\{v_{\epsilon}<\delta_0\}.
\end{equation*}
As we mentioned above, $v_{\epsilon}$ satisfies
\begin{equation*}
V_{\epsilon,\lambda_1} \leq v_{\epsilon} \leq V_{\epsilon,\lambda_2}
\end{equation*}
for some $0<\lambda_2\leq\lambda_1<\infty$. Hence the function $v_{\epsilon}$ is trapped in between $V_{\epsilon, \lambda_1}$ amd $V_{\epsilon, \lambda_2}$ near the $T_{a_{\epsilon}}$. Thus, there exists a uniform constant $c>1$ such that
\begin{equation*}
T_{a^{\ast}_{\epsilon}} \subset T_{\tilde{a}^{\ast}_{\epsilon}}
\subset T_{ca^{\ast}_{\epsilon}}.
\end{equation*}
Therefore, the hole
$T_{\tilde{a}^{\ast}_{\epsilon}}$ is not much different from
$T_{a^{\ast}_{\epsilon}}$. Since
$v_{\e}$ satisfies
\begin{equation}\label{eq-4.2.1}
0<c<v_{\epsilon}<C<\infty, \quad \mbox{in $\Omega_{\tilde{a}^{\ast}_{\e}}\times(0,T]$},
\end{equation}
by uniformly ellipticity of $v_{\epsilon}$, (\ref{eq-4.2.1}) has the harnack type
inequality. Following the same argument in Section
\ref{sec-2-1}(Heat Operator), we have the following Lemma.
\begin{lemm}
Set  $a_{\e}=(\frac{\epsilon a^*_{\e}}{2})^{1/2}$. Then
$$\osc_{\{B_{\e}(m)\backslash B_{a_{\e}}(m)\}\times[t_0-a^2_{\epsilon},t_0]}v_{\e}=o(\e^{\gamma})$$
for $m\in \e\Z\cap\supp\vp$ and for some $0<\gamma\leq 1$.
\end{lemm}
\subsection{Discrete Gradient Estimate}
$v_{\epsilon}$ can be approximated by the solutions, $v_{\epsilon,
\delta}$, of the following penalized equations, \cite{Fr}, for
sufficiently large number $M>0$,
\begin{equation}\label{eq-4-3-1}
\begin{aligned}
\La v_{\epsilon, \delta}-v^{\frac{1}{m}-1}_{\epsilon,
\delta}(v_{\epsilon, \delta})_t&+\beta_{\delta}(-v_{\epsilon,
\delta}+\delta+M\xi(x) )=0\quad \textrm{in}\,\, Q_T\\
&v_{\epsilon, \delta}=\delta \qquad \qquad  \qquad \qquad
\textrm{on} \,\,
\partial_l Q_T
\end{aligned}
\end{equation}
where $\beta_{\delta}(s)$ satisfies
\begin{equation*}
\begin{aligned}
\beta_{\delta}'(s) \geq 0,\quad &\beta_{\delta}''(s) \leq 0, \quad
\beta_{\delta}(0)=-1,\\
\beta_{\delta}(s)&=0 \quad \textrm{for}\,\, s>\delta,\\
\lim_{\delta \to 0}\beta_{\delta}(s) &\to -\infty \quad \textrm{for}
\,\, s<0.
\end{aligned}
\end{equation*}
Using this, we obtain the following results.
\begin{lemm}\label{lem-negative-of-v-epsilon-forever-00999}
If $(v_{\epsilon})_t$ is non-positive at $t=0$, then
$(v_{\e,\delta})_t\leq 0$ for all $t \in (0,T]$.
\end{lemm}
\begin{proof}
First, we assume
\begin{equation*}
(v_{\epsilon, \delta})_t(\cdot,0) <0.
\end{equation*}
 Since $v_{\epsilon,
\delta}$ is positive, we have
\begin{equation*}
\La (v_{\epsilon, \delta})_t-\left(\frac{1}{m}-1\right)v_{\epsilon,
\delta}^{\frac{1}{m}-2}(v_{\epsilon, \delta})_t^2-v_{\epsilon,
\delta}^{\frac{1}{m}-1}((v_{\epsilon, \delta})_t)_t-(v_{\epsilon,
\delta})_t\beta'_{\delta}(\cdot)=0.
\end{equation*}
Hence
\begin{equation*}
((v_{\epsilon, \delta})_t)_t=v_{\epsilon, \delta}^{1-\frac{1}{m}}\La
(v_{\epsilon,
\delta})_t+\left(1-\frac{1}{m}\right)\frac{(v_{\epsilon,
\delta})_t^2}{v_{\epsilon, \delta}}-v_{\epsilon,
\delta}^{1-\frac{1}{m}}(v_{\epsilon,
\delta})_t\beta'_{\delta}(\cdot).
\end{equation*}
Let $f_{\delta}(s)$ be a function having the maximum value of
$(v_{\epsilon, \delta})_t$ at $t=s$, then there exist points
$x(s)=(x_1(s),\cdots,x_n(s)) \in \mathbb{R}^n$ such that
\begin{equation*}
f_{\delta}(s)=(v_{\epsilon, \delta})_t(x(s),s).
\end{equation*}
Since $x(s)$ are maximum points, we have
\begin{equation*}
(f_{\delta})_s=((v_{\epsilon, \delta})_t)_s=((v_{\epsilon,
\delta})_t)_t+\nabla (v_{\epsilon, \delta})_t\cdot
x'(s)=((v_{\epsilon, \delta})_t)_t.
\end{equation*}
Hence
\begin{equation*}
(f_{\delta})_s\leq
\left(1-\frac{1}{m}\right)\frac{f_{\delta}^2}{v_{\epsilon,
\delta}}-v_{\epsilon,
\delta}^{1-\frac{1}{m}}f_{\delta}\beta'_{\delta}\leq C_{\e,\delta}
f_{\delta}
\end{equation*}
, which implies $$f_{\delta}(s)\leq f_{\delta}(0)e^{C_{\e,\delta}
s}<0.$$ When $(v_{\epsilon, \delta})_t(\cdot,0) \leq 0$, we can
approximate $(v_{\epsilon, \delta})(\cdot,0)$ by a smooth initial
data, $(v_{\epsilon, \delta,k})(\cdot,0)$ such that $(v_{\epsilon,
\delta,k})_t(\cdot,0)<0$. By the argument above, we know that
$(v_{\epsilon, \delta,k})_t(\cdot,s)<0$ and then $(v_{\epsilon,
\delta,k})(\cdot,s_1)>(v_{\epsilon, \delta,k})(\cdot,s_2)$ for
$s_1>s_2$. Since the operator is uniformly elliptic on each compact
subset $D$ of $\Omega_{\e}$, we have uniform convergence of
$v_{\epsilon, \delta,k}$ to $v_{\epsilon, \delta}$. Hence $(v_{\epsilon, \delta})(\cdot,s_1)\geq (v_{\epsilon,
\delta})(\cdot,s_2)$ for $s_1>s_2$ and lemma follows.
\end{proof}
\begin{lemm}\label{lem-grad-e}
If $\frac{d v_{\epsilon}}{dt}\big|_{t=0}$ is non-positive, then
\begin{equation*}
|\nabla v_{\epsilon,\delta}|_{L^{\infty}}\leq C_{\e}
\end{equation*}
with $C_{\e}$ satisfying $\lim_{\e \to 0}C_{\e}=\infty$.
\end{lemm}
\begin{proof}
For $i \in \{1,\cdots,n\}$, we will have
\begin{equation*}
\begin{aligned}
\La (v_{\epsilon,
\delta})_{x_i}&-\left(\frac{1}{m}-1\right)v_{\epsilon,
\delta}^{\frac{1}{m}-2}(v_{\epsilon, \delta})_{x_i}(v_{\epsilon,
\delta})_t\\
&-v_{\epsilon, \delta}^{\frac{1}{m}-1}\big((v_{\epsilon,
\delta})_{x_i}\big)_t-\beta'(\cdot)\big((v_{\epsilon,
\delta})_{x_i}-M\xi_{x_i}\big)=0.
\end{aligned}
\end{equation*}
Hence
\begin{equation*}
\begin{aligned}
\La \big(|(v_{\epsilon,
\delta})_{x_i}|^2\big)&-\left(\frac{1}{m}-1\right)v_{\epsilon,
\delta}^{\frac{1}{m}-2}|v_{\epsilon, \delta}|^2_{x_i}(v_{\epsilon,
\delta})_t\\
&-v_{\epsilon, \delta}^{\frac{1}{m}-1}\big((v_{\epsilon,
\delta})_{x_i}^2\big)_t-2\beta'(\cdot)\big(|(v_{\epsilon,
\delta})_{x_i}|^2-(v_{\epsilon, \delta})_{x_i}M\xi_{x_i}\big)\geq 0.
\end{aligned}
\end{equation*}
Let $X_i=\sup_{(x,t)\in Q_T}|(u_{\e ,\delta})_{x_i}|^2$
and assume that the maximum $X_i$ is achieved at $(x_0,t_0)$. Then we
have , at $(x_0,t_0)$,
\begin{equation*}
\La(u_{\e ,\delta})_{x_i}^2\leq 0 \qquad \mbox{and} \qquad ((u_{\e ,\delta})_{x_i}^2)_t\geq 0.
\end{equation*}
By the Lemma \ref{lem-negative-of-v-epsilon-forever-00999}, we get
\begin{equation*}
(v_{\epsilon,\delta})_t\leq 0.
\end{equation*}
Thus, if $X_i=|v_{\epsilon, \delta}|_{x_i}^2>|\xi_{x_i}|^2$ at an interior
point $(x_0,t_0)$, we can get a contradiction. Therefore $X_i\leq |\xi_{x_i}|^2$ in the interior of $Q_T$. To get a bound of the maximum
$X_i$ on the lateral boundary $\partial_l Q_T$ or at the
initial time, we consider the least super-solution $f$ of the obstacle
problem
\begin{equation*}
\begin{cases}
\begin{aligned}
\La f &\leq 0 \qquad \qquad \qquad \quad \mbox{in $\Omega$}\\
f(x) &\geq g(x) \qquad \qquad \quad \mbox{in $\Omega$} \\
f(x)& = \delta \qquad \qquad \qquad \quad \mbox{on $\partial
\Omega$}.
\end{aligned}
\end{cases}
\end{equation*}
Then $f$ is a stationary super-solution with $f>g$ in $\Omega$ and $f=v_{\epsilon, \delta}$ on $\partial \Omega$. Hence, by the maximum principle and Hopf principle, we get
\begin{equation*}
X_i \leq
C\big(\|\xi\|_{C^1(Q_T)}+\|g\|_{C^1(Q_T)}+\|f\|_{C^1(Q_T)})
\end{equation*}
and the lemma follows.
\end{proof}
\begin{lemm}\label{lem-d-g}
For each unit direction $e$, we define the difference quotient of
$v_{\epsilon}$ at $x$ in the direction $e$ by
\begin{equation*}
\Delta_e^{\epsilon}v_{\epsilon,\delta}=\frac{v_{\epsilon,\delta}(x+\epsilon
e,t)-v_{\epsilon,\delta}(x,t)}{\epsilon}.
\end{equation*}
If $\frac{d v_{\epsilon}}{dt}\big|_{t=0}$ is non-positive, then
\begin{equation*}
|\Delta_e^{\epsilon}v_{\epsilon,\delta}|\leq C
\end{equation*}
uniformly in $Q_T$.
\end{lemm}
\begin{proof}
Since $M\xi(x)$ is $\epsilon$-periodic, we will have
\begin{equation*}
\La \big(\Delta_e^{\epsilon}(v_{\epsilon,
\delta})\big)-\Delta_e^{\epsilon}\big(v_{\epsilon,
\delta}^{\frac{1}{m}-1}\big)(v_{\epsilon, \delta})_t-v_{\epsilon,
\delta}^{\frac{1}{m}-1}\big(\Delta_e^{\epsilon}(v_{\epsilon,
\delta})\big)_t-\Delta_e^{\epsilon}(v_{\epsilon,
\delta})\beta'(\cdot)=0.
\end{equation*}
Hence
\begin{equation*}
\begin{aligned}
\La \big(|\Delta_e^{\epsilon}(v_{\epsilon,
\delta})|^2\big)&-2\Delta_e^{\epsilon}(v_{\epsilon,
\delta})\Delta_e^{\epsilon}\big(v_{\epsilon,
\delta}^{\frac{1}{m}-1}\big)(v_{\epsilon, \delta})_t\\
&-v_{\epsilon,
\delta}^{\frac{1}{m}-1}\big(|\Delta_e^{\epsilon}(v_{\epsilon,
\delta})|^2\big)_t-2|\Delta_e^{\epsilon}(v_{\epsilon,
\delta})|^2\beta'(\cdot)\geq 0.
\end{aligned}
\end{equation*}
Since $\Delta_e^{\epsilon}(v_{\epsilon, \delta})$ and
$\Delta_e^{\epsilon}\big(v_{\epsilon, \delta}^{\frac{1}{m}-1}\big)$
have different sign, we can get a contradiction if
$|\Delta_e^{\epsilon}(v_{\epsilon, \delta})|^2$ has a maximum value
in the interior. Hence, 
\begin{equation*}
|\Delta_e^{\epsilon}v_{\epsilon}|^2<C \qquad int(Q_T)
\end{equation*}
for some constant $C>0$. On the lateral boundary, the estimate is obtained from the
Lemma \ref{lem-4}. Thus we get
$|\Delta_e^{\epsilon}v_{\epsilon}|<C$ in $Q_T$.
\end{proof}

\begin{corollary}
If $\frac{d v_{\epsilon}}{dt}\big|_{t=0}$ is non-positive, then we
have
$$(v_{\e})_t\leq 0$$
and
\begin{equation*}
|\Delta_e^{\epsilon}v_{\epsilon}|\leq C
\end{equation*}
uniformly in $Q_T$.
\end{corollary}
\begin{proof}
By Lemma \ref{lem-grad-e}, for each $\e>0$, $v_{\epsilon, \delta}$
converges uniformly to $v_{\epsilon}$ up to subsequence. Then
$v_{\epsilon, \delta}(x,t_1)\geq v_{\epsilon, \delta}(x,t_2)$ for
$t_1<t_2$ implies $v_{\epsilon}(x,t_1)\geq v_{\epsilon}(x,t_2)$ and
then $(v_{\epsilon})_t(x,t)\leq 0$. By the Lemma \ref{lem-d-g}, we have
\begin{equation*}
\left|\frac{v_{\epsilon,\delta}(x+\epsilon
e,t)-v_{\epsilon,\delta}(x,t)}{\epsilon}\right|<C.
\end{equation*}
Therefore, by taking $\delta\ra 0$, $|\Delta_e^{\epsilon}v_{\epsilon}|\leq
 C$.
\end{proof}
\subsection{Correctibility Condition II}
Likewise elliptic eigenvalue problem, we need an appropriate
corrector. Similar to the correctibility condition I, we
start with the following form
\begin{equation*}
\overline{w}_{\e}=d-dw_{\epsilon}
\end{equation*}
where $w_{\epsilon}$ is given by \eqref{eq-cases-definition-of-corrector-cite-by-CL-5} and for some constant $d>0$.
\begin{lemm}\label{lemma-correctibility-2---}
Let $\bar{k}_{c,d,\epsilon}$ be such that
\begin{equation*}
\begin{aligned}
\big((1-w_{\epsilon})^p&-1\big)d^pc+(d-dw_{\epsilon})^p\La
(d-dw_{\epsilon})\\
&=\big((1-w_{\epsilon})^p-1\big)d^pc-d^{1+p}(1-w_{\epsilon})^p\La
w_{\epsilon}=\bar{k}_{c,d,\epsilon}
\end{aligned}
\end{equation*}
for some $c,d>0$. Then, we have
\begin{equation*}
-d^{1+p}\kappa_{_{B_{r_0}}}=\overline{k}_{c,d}
\end{equation*}
where $\bar{k}_{c,d}=\lim_{\epsilon\to 0}\bar{k}_{c,d,\epsilon}$ and $\kappa_{_{B_{r_0}}}$ is the harmonic capacity of $B_{r_0}$.
\end{lemm}
\begin{proof}
Set
$v_{\epsilon}(x)=w_{\epsilon}(a^{\ast}_{\epsilon}x+m)$, then
$v_{\epsilon}$ satisfies
\begin{equation*}
\begin{cases}
\begin{aligned}
\big((1-w_{\epsilon})^p-1\big)d^pc&-d^{1+p}(1-v_{\epsilon})^p\La
v_{\epsilon}=\bar{k}_{c,d,\epsilon}(a^{\ast}_{\epsilon})^2 \qquad
\mbox{in
$Q^{^{\frac{\epsilon}{a^{\ast}_{\epsilon}}}}_0\backslash B_1$}\\
v_{\epsilon}&=1 \qquad \qquad \qquad \qquad \qquad \qquad \qquad
\quad \mbox{on $\partial
B_1$}\\
v_{\epsilon}=&|\nu\cdot\nabla v_{\epsilon}|=0 \qquad \qquad \qquad \qquad
\qquad \quad \mbox{on $\partial
Q^{^{\frac{\epsilon}{a^{\ast}_{\epsilon}}}}_0$}.
\end{aligned}
\end{cases}
\end{equation*}
Thus we get
\begin{equation*}
\begin{aligned}
(a^{\ast}_{\epsilon})^2d^pc\int_{_{Q^{^{\frac{\epsilon}{a^{\ast}_{\epsilon}}}}_0\backslash B_1}}&\big((1-v_{\epsilon})^p-1\big)dx=\\
&d^{1+p}\int_{_{Q^{^{\frac{\epsilon}{a^{\ast}_{\epsilon}}}}_0\backslash
B_1}}(1-v_{\epsilon})^p\La
v_{\epsilon}dx+(a^{\ast}_{\epsilon})^2\int_{_{Q^{^{\frac{\epsilon}{a^{\ast}_{\epsilon}}}}_0\backslash
B_1}}\bar{k}_{c,d,\epsilon}dx.
\end{aligned}
\end{equation*}
Similar to the correctibility condition I, Lemma \ref{lemma-correctibility-condition-1-0000000000}, letting $\epsilon\to 0$, we get
\begin{equation}\label{eq-4-4-1}
-d^{1+p}\kappa_{_{B_1}}=\lim_{\e \to
0}\big[\bar{k}_{c,d,\e}(a^{\ast}_{\e})^2(\frac{\e}{a^{\ast}_{\e}})^{n}\big]=\frac{1}{r^{n-2}_0}\overline{k}_{c,d}
\end{equation}
where $\kappa_{_{B_{1}}}$ is the harmonic capacity of $B_{1}$ and $\bar{k}_{c,d}=\lim_{\epsilon\to 0}\bar{k}_{c,d,\epsilon}$ since
$\widehat{w}_{\e}=(1-w_{\e})\rightharpoonup 1$ in
$L^2(\mathbb{R}^n)$. If we multiply equation (\ref{eq-4-4-1}) by
$r^{n-2}_0$, we obtain
\begin{equation*}
-d^{1+p}\kappa_{_{B_{r_0}}}=\overline{k}_{c,d}
\end{equation*}
where $\kappa_{_{B_{r_0}}}$ is the harmonic capacity of $B_{r_0}$.
\end{proof}
\subsection{Homogenized Equation}
Finally, we show the homogenized equation satisfied by the
limit $u$ of $u_{\epsilon}$ through viscosity methods.
\begin{theo}\label{thm-ii}
Let $a^{\ast}_{\epsilon}=\epsilon^{{\alpha}_{\ast}}$ for
${\alpha}_{\ast}=\frac{n}{n-2}$ for $n \geq 3$ and
$a^{\ast}_{\epsilon}=e^{-\frac{1}{\epsilon^2}}$ for $n=2$. Then for
$c_0a_{\epsilon}^{\ast} \leq a_{\epsilon} \leq
C_0a_{\epsilon}^{\ast}$, $v$ is a viscosity solution of
\begin{equation*}
\begin{cases}
\begin{aligned}
v^{1-\frac{1}{m}}(\La v&-\kappa_{_{B_{r_0}}}v_+)-v_t=0 \qquad \qquad
\mbox{in
$Q_T$}\\
v&=0 \qquad \qquad \qquad\qquad \qquad \mbox{on $\partial_l Q_T$}\\
v&=g^m \qquad \qquad \qquad \qquad \quad \mbox{in $\Omega \times
\{t=0\}$}
\end{aligned}
\end{cases}
\end{equation*}
where $\kappa_{_{B_{r_0}}}$ is the capacity of $B_{r_0}$ if
$r_0=\lim_{\epsilon \to
0}\frac{{\alpha}_{\epsilon}}{{\alpha}^{\ast}_{\epsilon}}$ exists.
\end{theo}
\begin{proof}
For $\epsilon>0$,
\begin{equation*}
v_{\epsilon}^{1-\frac{1}{m}}(\La v_{\epsilon}-\kappa_{_{B_{r_0}}}
v_{\epsilon})-(v_{\epsilon})_t=-v_{\epsilon}^{1-\frac{1}{m}}\kappa_{_{B_{r_0}}}
v_{\epsilon}\leq 0.
\end{equation*}
Thus, the limit $v$ of $v_{\epsilon}$ also satisfies
\begin{equation*}
v^{1-\frac{1}{m}}(\La v-\kappa_{_{B_{r_0}}} v_+)-v_t\leq 0
\end{equation*}
in a viscosity sense. So we are going to show that $v$ is a
sub-solution. Let us assume that there is a parabola $P$ touching $v$
from above at $x_0$ and
\begin{equation*}
P^{1-\frac{1}{m}}(\La P-\kappa_{_{B_{r_0}}} P)-P_t\leq -2\delta_0<0.
\end{equation*}
In a small neighborhood of $x_0$, $B_{\eta}(x_0)\times
[t_0-\eta^{2},t_0]$, we can choose another parabola $Q$ such that
\begin{equation*}
\begin{cases}
D^2Q > D^2P \qquad \qquad \qquad \qquad \mbox{in $B_{\eta}(x_0)\times[t_0-\eta^{2},t_0]$} \\
Q_t<P_t \qquad \qquad \qquad \qquad \qquad \mbox{in $B_{\eta}(x_0)\times[t_0-\eta^{2},t_0]$}\\
Q(x_0,t_0)<P(x_0,t_0)-\delta \\
Q(x,t)>P(x,t)\qquad \qquad \qquad \textrm{on}\,\,
\begin{array}{c}\{\partial B_{\eta}(x_0)\times
[t_0-\eta^{2},t_0]\}\\ \cap
\{B_{\eta}(x_0)\times\{t_0-\eta^{2}\}\}\end{array}.
\end{cases}
\end{equation*}
and 
\begin{equation*}
Q_1^{1-\frac{1}{m}}(\La Q-\kappa Q_1)-Q_t\leq -\delta_0<0
\end{equation*}
for $Q_1=Q(x_1,t_1)=\min_{B_{\eta}(x_0)\times
[t_0-\eta^{2},t_0]}Q(x,t)$. Let us consider
\begin{equation*}
Q_{\epsilon}(x,t)=Q(x,t)- Q_1w_{\epsilon}(x)+\epsilon_0+h(x,t)
\end{equation*}
for a small number $0<{\epsilon}_0<\frac{\delta}{4}$ and a function
$h(x,t)$ we choose later. In $\{B_{\eta}(x_0)\cap
\Omega_{a_{\epsilon}}\}\times[t_0-{\eta}^{2},t_0]$, $Q_{\epsilon}$ satisfies
\begin{equation*}
\begin{aligned}
Q^{1-\frac{1}{m}}_{\epsilon}&\La
Q_{\epsilon}-(Q_{\epsilon})_t\leq\\
&\bigg[Q_1^{1-\frac{1}{m}}(1-w_{\epsilon})^{1-\frac{1}{m}}\La
Q-Q_1^{2-\frac{1}{m}}(1-w_{\epsilon})^{1-\frac{1}{m}}\La
w_{\epsilon}-Q_t\bigg]\\
&+\bigg[c(Q-Q_1+\epsilon_0+h)^{1-\frac{1}{m}}\La
Q+(Q-Q_1w_{\epsilon}+\epsilon_0+h)^{1-\frac{1}{m}}\La
h\\
&\qquad -h_t\bigg]:=[1]+[2]
\end{aligned}
\end{equation*}
with $c=0$ if $\La Q< 0$ and $c=1$ if $\La Q\geq 0$. To
remove the $[2]$, we consider the following initial value problem
\begin{equation*}
\begin{cases}
\begin{aligned}
a^{ij}(x,t)D_{ij}\tilde{h}&-\tilde{h}_t=f(x,t) \qquad \qquad
\mbox{in $\mathbb{R}^n\times
(0,\infty)$}\\
\tilde{h}&\geq 0 \qquad \qquad \qquad \qquad  \mbox{in
$\mathbb{R}^n\times
(0,\infty)$}\\
\tilde{h}(x,0)&=Q_1w_{\epsilon}(x)
\end{aligned}
\end{cases}
\end{equation*}
with
\begin{equation*}
\begin{aligned}
a^{ij}(x,t)&=\begin{cases}\qquad 0 \qquad  \qquad
\qquad \qquad \qquad \qquad \quad \mbox{if $i\neq j$}\\
\big[(Q-Q_1)\zeta(x,t)+\epsilon_0+\tilde{h}\big]^{1-\frac{1}{m}} \qquad \mbox{otherwise}\end{cases},\\
f(x,t)&=-c\big[(Q-Q_1)\zeta(x,t)+\epsilon_0+\tilde{h}\big]^{1-\frac{1}{m}}\La
Q,\\
\zeta(x,t)\in C^{\infty},&\quad 0\leq \zeta(x,t) \leq 1, \quad
\zeta(x,t)=1 \quad \mbox{in
$B_{\eta}(x_0)\times[0,\eta^2]$} \qquad \\
 \mbox{and}&\quad
\zeta(x,t)=0 \quad \mbox{in
$\{B_{\eta+\eta^2}\times[0,(\eta+\eta^2)^2]\}^c$}.
\end{aligned}
\end{equation*}
Since the equation has non-degenerate coefficients, we can find the
solution $\tilde{h}(x,t)$ of the initial value problem. We can also
observe the fact that the solution $\tilde{h}(x,t)$ decays rapidly
in a small time because $w_{\epsilon} \rightharpoonup 0$ as
$\epsilon \to 0$ in $H_0^1(\mathbb{R}^n)$. Hence, for sufficiently small $\epsilon>0$, we get
\begin{equation*}
0 \approx \tilde{h}(x,t)<\frac{\delta}{4} \qquad\mbox{at $\quad
t=\eta^2$}.
\end{equation*}
Therefore, $Q_{\epsilon}$ satisfies
\begin{equation*}
\begin{aligned}
Q^{1-\frac{1}{m}}_{\epsilon}\La Q_{\epsilon}-(Q_{\epsilon})_t \leq &
Q_1^{1-\frac{1}{m}}\La Q+
Q_1^{1+\frac{1}{m}}\big[(1-w_{\epsilon}^{1-\frac{1}{m}}-1)\big]\La
Q\\
&-Q_1^{2-\frac{1}{m}}(1-w_{\epsilon})^{1-\frac{1}{m}}\La
w_{\epsilon}-Q_t
\end{aligned}
\end{equation*}
in $\{B_{\eta}(x_0)\cap
\Omega_{a_{\epsilon}}\}\times[t_0-{\eta}^{2},t_0]$. By correctibility condition II, Lemma \ref{lemma-correctibility-2---},
\begin{equation*}
\begin{aligned}
Q^{1-\frac{1}{m}}_{\epsilon}\La Q_{\epsilon}-(Q_{\epsilon})_t\leq & Q_1^{1-\frac{1}{m}}\La Q+\bar{k}_{_{\La Q,Q_1,\epsilon}}-Q_t\\
&\leq Q_1^{1-\frac{1}{m}}\La Q+\bar{k}_{_{\La Q,Q_1}}+\frac{\delta_0}{2}-Q_t\\
&\leq Q_1^{1-\frac{1}{m}}\big(\La Q-\kappa_{_{B_{r_0}}}
Q_1\big)+\frac{\delta_0}{2}-Q_t\leq -\frac{\delta_0}{2}<0
\end{aligned}
\end{equation*}
for small $\epsilon>0$. Hence $ Q^{1-\frac{1}{m}}_{\epsilon}\La
Q_{\epsilon}-(Q_{\epsilon})_t<0$ and $Q_{\epsilon}\geq u_{\epsilon}$
on $\partial \{B_{\rho}(x_0)\cap \Omega_{a_{\epsilon}}\}\times
[t_0-\rho^{2},t_0]$ and $\{B_{\rho}(x_0,t_0)\cap
\Omega_{a_{\epsilon}}\}\times\{t_0-\rho^{2}\}$ for some $\rho>0$. By a comparison principle, $Q_{\epsilon}(x_0,t_0)\geq
u_{\epsilon}(x_0,t_0)$ and then $Q(x_0,t_0)+\frac{\delta}{2} \geq
u(x_0,t_0)$. On the other hand,
$Q(x_0,t_0)<P(x_0,t_0)-\delta<u(x_0,t_0)-\delta_0$, which is a
contradiction.
\end{proof}

{\bf Acknowledgement}
Ki-Ahm Lee is supported by the Korea Research Foundation Grant
funded by the Korean Government (MOEHRD) (KRF-2005-041-C00040)


\begin{thebibliography}{99}
\bibitem[A]{A} Allaire, Gr\'egoire {\it Homogenization and two-scale convergence.}  SIAM J. Math. Anal.  23  (1992),  no.{\bf 6}, 1482--1518.
\bibitem[AP]{AP}Attouch, H\'edy and Picard, Colette,
{\it  Variational inequalities with varying obstacles: the general
form of the limit problem,} J. Funct. Anal., Vol {\bf 50}, 3, 1983,
329-386.
\bibitem[BCR]{BCR} Baffico, L.; Conca, C.; Rajesh, M. {\it Homogenization of a class of nonlinear eigenvalue problems.}  Proc. Roy. Soc. Edinburgh Sect. A  136  (2006),  no. 1, 7--22
\bibitem[BLP]{BLP} Bensoussan, Alain; Lions, Jacques-Louis; Papanicolaou, George {\it Asymptotic analysis for periodic structures.} Studies in Mathematics and its Applications, {\bf 5}. North-Holland Publishing Co., Amsterdam-New York, 1978. xxiv+700 pp. ISBN: 0-444-85172-0
\bibitem[Ca]{Ca} Caffarelli, L. A.
{\it A note on nonlinear homogenization.} Comm. Pure Appl. Math.
{\bf 52} (1999), no. 7, 829--838.
\bibitem[CC]{CC} L. Carbone; F. Colombini
{\it On convergence of functionals with unilateral constraints} J.
Funct. Anal., {\bf 15}, (1983), 329-386
\bibitem[CD]{CD} Casado-D\'iaz J,{\it Two-scale convergence for nonlinear Dirichlet problems in perforated domains,} Proc. R. Soc. Edinburgh A130 (2000) 249--276
\bibitem[CL]{CL}Luis Caffarelli; Ki-ahm Lee,
{\it Viscosity Method for Homogenization of Highly Oscillating
Obstacles}  Indiana Univ. Math. J.  57 (2008),  1715-1742.
\bibitem[CL1]{CL1} Caffarelli, L.; Lee, Ki-Ahm
{\it Homogenization of nonvariational viscosity solutions.} Rend.
Accad. Naz. Sci. XL Mem. Mat. Appl. (5) 29 (2005), no. 1, 89--100.
\bibitem[CL2]{CL2} Caffarelli, L.; Lee, K. {\it Homogenization of oscillating free boundaries: the elliptic case.}  Comm. Partial Differential Equations  {\bf 32}  (2007),  no. 1-3, 149--162.
\bibitem[CLM]{CLM} Caffarelli, Luis A.; Lee, Ki-Ahm; Mellet, Antoine {\it Singular limit and homogenization for flame propagation in periodic excitable media.}  Arch. Ration. Mech. Anal.  {\bf 172}  (2004),  no. 2, 153--190.
\bibitem[CLM1]{CLM1} Caffarelli, L. A.; Lee, K.-A.; Mellet, A. {\it Homogenization and flame propagation in periodic excitable media: the asymptotic speed of propagation.}  Comm. Pure Appl. Math. {\bf 59}  (2006),  no. 4, 501--525.
\bibitem[CLM2]{CLM2} Caffarelli, L. A.; Lee, K.-A.; Mellet, A. Flame propagation in one-dimensional stationary ergodic media.  Math. Models Methods Appl. Sci.  17  (2007),  no. 1, 155--169.
\bibitem[CM]{CM} Cioranescu, Doina; Murat, Francois
{\it A strange term coming from nowhere} Topics in the mathematical
modelling of composite materials, 45--93, Progr. Nonlinear
Differential Equations Appl., 31, Birkh\"auser Boston, Boston, MA,
1997.
\bibitem[CSW]{CSW} Caffarelli, Luis A.; Souganidis, Panagiotis E.; Wang, L. {\it Homogenization of fully nonlinear, uniformly elliptic and parabolic partial differential equations in stationary ergodic media.}  Comm. Pure Appl. Math.  {\bf 58}  (2005)
\bibitem[D1]{D1} G. Dal Maso
{\it  Limiti di soluzioni di problemi varizionali con ostacoli
bilaterali} Atti. Accad. Naz. Lincei, Rend. Cl. Sci. Fis. Mat.
Natur.,{\bf 69},(1980),333-337
\bibitem[D2]{D2} G. Dal Maso
{\it Asymptotic behavior of minimum problems with bilateral
obstacles} Ann. mat. pura ed appl., {\bf 129},(1981),327-366
\bibitem[DG]{DG} De Giorgi, Ennio {\it $G$-operators and $\Gamma$-convergence.}  Proceedings of the International Congress of Mathematicians, Vol. 1, 2 (Warsaw, 1983),  1175--1191, PWN, Warsaw, 1984.
\bibitem[DL]{DL} G. Dal Maso; P. Longo
{\it $\Gamma$ -limits of obstacles} Ann. mat. pura ed appl., {\bf
128}, (1981),1-50
\bibitem[DDL]{DDL}E. De Diorgi;G. Dal Maso; P. Longo
{\it
 $\Gamma$ -limiti di ostacoli}Atti. Acad. Naz. Lincei, Rend. Cl. Sci. Fis. Mat. Natur.,{\bf 68}. (1980),481-487
\bibitem[Ev]{Ev} L.C. Evans (1998) {\it Partial Differential Equations} Graduate Studies in Mathematics, Vol. 19. American Mathematical Society, Probidence, RI.
\bibitem[Ev1]{Ev1} Evans, L. C. Periodic {\it homogenisation of certain fully nonlinear partial differential equations.} Proc. Roy. Soc. Edinburgh Sect. A {\bf 120} (1992), no. 3-4, 245--265.
\bibitem[Fr]{Fr} A. Friedman
{\it Variational Principles and Free-boundary problems}, Robert E.
Krieger Publishing Company, 1988
\bibitem[GT]{GT} {\sc D. Gilbarg and N.S. Trudinger},
\textit{``Elliptic Partial Differential Equations of Second Order''}, GMW
\textbf{224}, 1977, Springer Verlag, Heidelberg.
\bibitem[JKO]{JKO} Jikov, V. V.; Kozlov, S. M.; Oleinik, O. A. {\it Homogenization of differential operators and integral functionals.} Translated from the Russian by G. A. Yosifian [G. A. Iosifʹ  yan]. Springer-Verlag, Berlin, 1994. xii+570 pp. ISBN: 3-540-54809-2
\bibitem[L1]{L1} Ki-Ahm Lee
{\it Obstacle Problem for Nonlinear $2^{nd}$-Order Elliptic
Operator} Ph-Thesis, New York University, 1998
\bibitem[L2]{L2} Ki-Ahm Lee
{\it The Obstacle Problem for Monge-Amp\'ere Equation} Comm. in
P.D.E.Volume {\bf 26}, 1\&2.,33-42,2001
\bibitem[LB]{LB} G.M. Lieberman, {\it Second Order Parabolic Partial
Differential Equations}, World Scientific, 1996.
\bibitem[N]{N} Nguetseng, Gabriel
{\it A general convergence result for a functional related to the
theory of homogenization.} SIAM J. Math. Anal. {\bf 20} (1989), no.
3,
\bibitem[NR]{NR} Nandakumaran, A. K.; Rajesh, M. {\it Homogenization of a parabolic equation in perforated domain with Dirichlet boundary condition.}  Proc. Indian Acad. Sci. Math. Sci.  112  (2002),  no. 3, 425--439.
\bibitem[S]{S} Spagnolo, Sergio {\it Convergence in energy for elliptic operators.}  Numerical solution of partial differential equations, III (Proc. Third Sympos. (SYNSPADE), Univ. Maryland, College Park, Md., 1975),  pp. 469--498. Academic Press, New York, 1976.
\bibitem[T]{T} Tartar, L. {\it Compensated compactness and applications to partial differential equations.}  Nonlinear analysis and mechanics: Heriot-Watt Symposium, Vol. IV,  pp. 136--212, Res. Notes in Math., {\bf 39}, Pitman, Boston, Mass.-London, 1979.
\bibitem[T1]{T1} Tartar, Luc {\it Topics in nonlinear analysis.} Publications Mathematiques d'Orsay 78, 13. University de Paris-Sud, Departement de Mathematique, Orsay, 1978. ii+271 pp.
\end{thebibliography}
\end{document}